\documentclass{amsart}

\usepackage{amsmath, amssymb, amsthm}
\usepackage[margin=1in]{geometry}
\usepackage{url}
\usepackage{bm}
\usepackage{dsfont}
\usepackage{relsize}
\usepackage{multicol}

\usepackage{hyperref}

\hypersetup{
     colorlinks=true,
     linkcolor=blue,
     filecolor=magenta,
     urlcolor=cyan,
     citecolor=red,
     pdfpagemode=FullScreen
     }

\usepackage{cleveref}

\theoremstyle{plain}
\newtheorem{thm}{Theorem}
\numberwithin{thm}{subsection}

\newtheorem{lem}[thm]{Lemma}
\newtheorem{prop}[thm]{Proposition}
\newtheorem{cor}[thm]{Corollary}
\newtheorem{rmk}[thm]{Remark}
\newtheorem*{ack}{Acknowledgement}

\theoremstyle{definition}
\newtheorem{defn}[thm]{Definition}

\newtheorem{ex}[thm]{Example}
\newtheorem{exs}[thm]{Examples}

\usepackage{tikz}
\usetikzlibrary{hobby, arrows.meta,bending, decorations.markings}
\tikzstyle{V}=[draw, fill =black, circle, inner sep=0pt, minimum size=3pt]

\begin{document}

\title{A Centraliser Analogue to the Farahat-Higman Algebra}

\author{Samuel Creedon}

\date{}

\maketitle

\begin{abstract}
We define a family of algebras $\mathsf{FH}_{m}$ which generalise the Farahat-Higman algebra introduced in \cite{FH59} by replacing the role of the center of the group algebra of the symmetric groups with centraliser algebras of symmetric groups. These algebras have a basis indexed by marked cycle shapes, combinatorial objects which generalise proper integer partitions. We analyse properties of marked cycle shapes and of the algebras $\mathsf{FH}_{m}$, demonstrating that some of the former govern the latter. The main theorem of the paper proves that the algebra $\mathsf{FH}_{m}$ is isomorphic to the tensor product of the degenerate affine Hecke algebra with the algebra of symmetric functions.
\end{abstract}

%%%%%%%%%%%%%%%%    INTRODUCTION   %%%%%%%%%%%%%%%%%%%%%%%%%%%%%%%%%%%%%%%%%%%%%%%%%%%%%%%%%%

\section{Introduction}

For $n\in\mathbb{Z}_{\geq 0}$, let $\mathfrak{S}_{n}$ denote the symmetric group of permutations of $[n]:=\{1,\dots,n\}$. In \cite{FH59}, Farahat and Higman constructed an algebra, denoted here by $\mathsf{FH}_{0}$, allowing one to analyse the centers $Z(\mathbb{Z}\mathfrak{S}_{n})$ for all $n$ simultaneously. The structure of $\mathsf{FH}_{0}$ was systematically studied and provided new results for the centers $Z(\mathbb{Z}\mathfrak{S}_{n})$ which led to an alternative proof of Nakayama's Conjecture regarding the $p$-blocks of $\mathfrak{S}_{n}$. We summarise here the core features of their work in a manner easily comparable to the results of this paper.

Permutations belong to the same conjugacy class if and only if they share the same cycle type. Dropping trivial cycles from cycle types induces a natural bijection between the set $\mathcal{P}^{\text{pr}}_{n}$ of proper (contain no part equal to 1) integer partitions of size no greater than $n$ and the set of  conjugacy classes of $\mathfrak{S}_{n}$. Let $\mathsf{CL}_{n}$ denote this bijection. As an example $\mathcal{P}^{\text{pr}}_{4}=\left\{ \emptyset, (2), (2,2), (3), (4) \right\}$ and $\mathsf{CL}_{4}((2,2))$ is the conjugacy class of $\mathfrak{S}_{4}$ consisting of the permutations obtained by the product of two disjoint transpositions. Working with proper partitions instead of partitions themselves allows for the indexing of conjugacy classes in a uniform manner as $n$ ranges over the nonnegative intergers. For $\lambda\in\mathcal{P}^{\text{pr}}_{n}$, let $K_{n}(\lambda)$ denote the formal sum of permutations in $\mathsf{CL}_{n}(\lambda)$ viewed as an element of the $\mathbb{Z}$-algebra $\mathbb{Z}\mathfrak{S}_{n}$. Such class-sums form a $\mathbb{Z}$-basis for the center $Z(\mathbb{Z}\mathfrak{S}_{n})$, and so for any $\lambda,\mu\in\mathcal{P}_{n}^{\text{pr}}$
\begin{equation} \label{CenterStructureConstants}
K_{n}(\lambda)K_{n}(\mu)=\sum_{\nu\in\mathcal{P}_{n}^{\text{pr}}}a_{\lambda,\mu}^{\nu}(n)K_{n}(\nu), 
\end{equation}
with $a_{\lambda,\mu}^{\nu}(n)\in\mathbb{Z}$ the corresponding structure constants. These constants were proved to be polynomial in $n$ by \cite[Theorem 2.2]{FH59}. In other words, let $t$ be a formal variable and $R:=\mathsf{Int}[t]$ the subalgebra of $\mathbb{Q}[t]$ consisting of all integer-valued polynomials, that is all polynomials $p(t)\in\mathbb{Q}[t]$ such that $p(\mathbb{Z})\subseteq\mathbb{Z}$. Then it was shown that there exists, for each triple of proper partitions $\lambda,\mu$, and $\nu$, a polynomial $f_{\lambda,\mu}^{\nu}(t)\in R$ such that $f_{\lambda,\mu}^{\nu}(n)=a_{\lambda,\mu}^{\nu}(n)$ for all $n\in\mathbb{Z}_{\geq 0}$. Then Farahat and Higman defined an $R$-algebra via the formal $R$-basis $\{K(\lambda) \ | \ \lambda\in\mathcal{P}^{\text{pr}}:=\bigcup_{n\geq 0}\mathcal{P}_{n}^{\text{pr}}\}$ and product given by
\begin{equation} \label{CenterStructurePolynomials}
K(\lambda)K(\mu)=\sum_{\nu\in\mathcal{P}^{\text{pr}}}f_{\lambda,\mu}^{\nu}(t)K(\nu).
\end{equation}
This is the algebra $\mathsf{FH}_{0}$, and it is obtained from the centers $Z(\mathbb{Z}\mathfrak{S}_{n})$ by replacing the structure constants $a_{\lambda,\mu}^{\nu}(n)$ with the corresponding polynomials $f_{\lambda,\mu}^{\nu}(t)$, and by replacing the class-sums $K_{n}(\lambda)$ with the formal symbols $K(\lambda)$, which can be thought of as infinite analogues. Comparing \emph{\Cref{CenterStructureConstants}} with \emph{\Cref{CenterStructurePolynomials}} it is immediate that we have surjective ring homomorphisms $\mathsf{pr}_{n}:\mathsf{FH}_{0}\rightarrow Z(\mathbb{Z}\mathfrak{S}_{n})$ for each $n\in\mathbb{Z}_{\geq 0}$ where $t\mapsto n$ and $K(\lambda)\mapsto K_{n}(\lambda)$ (where $K_{n}(\lambda)=0$ if $|\lambda|>n$). Moreover, given $X,Y\in\mathsf{FH}_{0}$, one can show that $\mathsf{pr}_{n}(X)=\mathsf{pr}_{n}(Y)$ for all $n\in\mathbb{Z}_{\geq 0}$ if and only if $X=Y$. Hence structural information of the algebra $\mathsf{FH}_{0}$ equates to structural information for the centers $Z(\mathbb{Z}\mathfrak{S}_{n})$ for all $n\in\mathbb{Z}_{\geq 0}$ uniformly. Numerous results for $\mathsf{FH}_{0}$ were given in \cite{FH59}, many of which were strongly connected to the combinatorics of proper partitions. The main result was \cite[Theorem 2.5]{FH59} which described a family of generators $E_{1},E_{2},\dots$ for $\mathsf{FH}_{0}$, the non-zero images of such under $\mathsf{pr}_{n}$ providing generators for the centers $Z(\mathbb{Z}\mathfrak{S}_{n})$ for each $n\in\mathbb{Z}_{\geq 0}$. Then these new central generators were used to describe central characters in positive characteristic, and from previous results of Farahat this led to an alternative proof of Nakayama's Conjecture.

Many of the results of \cite{FH59} have been recast in a more modern light. An illuminating retelling is given in \cite[Section 3]{Ryba21} which focuses on the role played by the Jucys-Murphy elements $L_{1},\dots, L_{n}$ of $\mathbb{Z}\mathfrak{S}_{n}$, which were at the center of the approach to the representation theory of $\mathfrak{S}_{n}$ taken in \cite{OV96}. Notably it was shown in \cite{Jucys74} that $\mathsf{pr}_{n}(E_{k})=e_{k}(L_{1},\dots, L_{n})$ where $e_{k}$ is the $k$-th elementary symmetric polynomial in $n$ commuting variables. Thus proving that the center $Z(\mathbb{Z}\mathfrak{S}_{n})$ is the subalgebra of symmetric polynomials in the corresponding Jucys-Murphy elements. Ryba used the Jucys-Murphy elements and their relation to $\mathsf{FH}_{0}$ to give an alternative approach to describing the central characters in postive characteristic to what was done in \cite{FH59}, which again led to a proof of Nakayama's Conjecture. From this new perspective it was proved in \cite[Theorem 3.8]{Ryba21} that we have an isomorphism $\mathsf{FH}_{0}\cong R\otimes \mathsf{Sym}$ of $R$-algebras, where $\mathsf{Sym}$ is the $\mathbb{Z}$-algebra of symmetric functions. This isomorphism maps the generators $E_{k}$ of $\mathsf{FH}_{0}$ to the $k$-th elementary symmetric functions, allowing one to interpret the elements $E_{k}$ as the ``evaluation'' of elementary symmetric functions at the Jucys-Murphy elements $L_{1}, L_{2},\dots$.

Beyong providing results for the centers $Z(\mathbb{Z}\mathfrak{S}_{n})$ and for the modular representation theory of $\mathfrak{S}_{n}$, the Farahat-Higman algebras $\mathsf{FH}_{0}$ have found other applications. For example, in \cite[p.131-132]{Mac95} a certain associated graded algebra $\mathsf{G}$ of $\mathsf{FH}_{0}$ is considered, and is regarded as a $\mathbb{Z}$-algebra as only constant polynomials are retained in the filtration. This $\mathbb{Z}$-algebra is isomorphic to the algebra of symmetric functions, and has been studied in relation to enumeration problems for permutation factorisation, see for example \cite{GJ94}. Analogous algebras for $\mathsf{FH}_{0}$ and $\mathsf{G}$ have also been defined by replacing the group $\mathfrak{S}_{n}$ with other finite groups. For example analogues in the setting of the spin symmetric group algebras have been given in \cite{TW09}, of wreath products of the symmetric groups with any finite group in \cite{W03} and \cite{Ryba21}, and of the general linear group over finite fields in \cite{WW19}. On a more subtle note, the work of Farahat and Higman demonstrated the idea of translating uniform polynomial behaviour of the structure of objects into a single universal object to study. This concept is at the heart of forming Deligne's category in \cite{Del04}, where it is using the fact the constants appearing from the compositions of morphisms between tensor spaces are polynomial. A detailed commentary from this perspective can be found in \cite[Section 1.1]{HS22}.

In this paper we define a family of $R$-algebras $\mathsf{FH}_{m}$ indexed by nonnegative integers $m$. These generalise the Farahat-Higman algebra of \cite{FH59} which is recovered by setting $m=0$, and is a subalgebra for each $\mathsf{FH}_{m}$. Instead of replacing the group $\mathfrak{S}_{n}$ with another finite group, we construct the algebras $\mathsf{FH}_{m}$ by instead replacing the role of the center $Z(\mathbb{Z}\mathfrak{S}_{n})$ with certain centraliser algebras. For any $m\leq n$ let $\mathsf{Stab}_{n}(m)$ denote the subgroup of $\mathfrak{S}_{n}$ consisting of permutations which fix the elements $1,2,\dots,m$. Then we work with the centraliser algebra $Z_{n,m}$ consisting of all elements of $\mathbb{Z}\mathfrak{S}_{n}$ which commute with $\mathsf{Stab}_{n}(m)$. When $m=0$ then $Z_{n,0}$ is simply the center $Z(\mathbb{Z}\mathfrak{S}_{n})$. The algebra $Z_{n,m}$ has a basis consisting of orbit-sums which are indexed by a certain subset $\Lambda_{\leq n}(m)$ of the set of $m$-marked cycle shapes (see \emph{\Cref{Defn:mMCS}} for a formal definition). These generalise proper integer partitions, where $\Lambda_{\leq n}(0)$ is naturally in bijection with $\mathcal{P}_{n}^{\text{pr}}$. The first main result of this paper is \emph{\Cref{Thm:PolyStructureConstants}} which proves that the structure constants associated to the orbit-sum basis of $Z_{n,m}$ are polynomial in $n$, with such polynomials belonging to $R$. In an analogous manner to \cite{FH59} we then define $\mathsf{FH}_{m}$ by a formal $R$-basis and product obtained by replacing structure constants with their corresponding polynomials (see \emph{\Cref{Defn:FHAlgebra}} and \emph{\Cref{Prop:FHisRAlg}}). We analyse the set of $m$-marked cycle shapes and also equip such a set with a product, turning it into a monoid, and a degree function. We show that this degree function translates into a filtration on $\mathsf{FH}_{m}$, and we show in \emph{\Cref{Prop:LeadingTerm}} that the monoid product, and certain combinatorial statistics associated to $m$-marked cycle shapes, govern the leading term of the product of basis elements in $\mathsf{FH}_{m}$. We then utilise this result to prove \emph{\Cref{Thm:FHmIsomorphicm}}, the main theorem of the paper, which establishes an isomorphism $\mathsf{FH}_{m}\cong R\otimes_{\mathbb{Z}}(\mathcal{H}_{m}\otimes \mathsf{Sym})$ of $R$-algebras, where $\mathcal{H}_{m}$ is the degenerate affine Hecke algebra. This generalises the isomorphism $\mathsf{FH}_{0}\cong R\otimes \mathsf{Sym}$ given by \cite[Theorem 3.8]{Ryba21}, and highlights that the additional combinatorial information that $m$-marked cycle shapes have over proper integer partitions translates algebraically to the extension of $\mathsf{Sym}$ by tensoring with $\mathcal{H}_{m}$. We conclude by summarising various results for the centraliser algebras $Z_{n,m}$ obtained from analogous results for $\mathsf{FH}_{m}$ in a uniform manner.

 The motivation behind constructing such algebras was in studying the affine partition algebra $\mathcal{A}_{2k}^{\text{aff}}(z)$ defined in \cite{CD22}. Notably the algebras $\mathsf{FH}_{m}$ emerged from a desire to understand the image of the action of $\mathcal{A}_{2k}^{\text{aff}}(z)$ on the tensor space $M\otimes V^{\otimes k}$ with $M$ any $\mathbb{C}\mathfrak{S}_{n}$-module and $V$ the permutation module induced from the action set $[n]$ (see \cite[Theorem 3.3.2]{CD22}). This connection will appear in a future paper. Also, effort was made for all the results of $\mathsf{FH}_{m}$ established in this paper to work in the integral setting, that is over the ring $R$. This keeps the algebra $\mathsf{FH}_{m}$ open as a potential tool to analyse the modular representation theroy of the centraliser algebras $Z_{n,m}$, which is an active area of research (see for example \cite{DEM13}).

The structure of this paper is as follows: \emph{Section 2} proves some technical results describing the cardinalities of certain orbits. \emph{Section 3} uses such results to prove the polynomial property of the structure constants in the centraliser algebras $Z_{n,m}$. We also set up some notation including that of an $m$-marked cycle shape and its degree. In \emph{Section 4} we define the $R$-algebras $\mathsf{FH}_{m}$ and prove that $\mathsf{R}\mathfrak{S}_{m}$ and $\mathsf{FH}_{0}$ are subalgebras. In \emph{Section 5} we describe a monoid structure on the set of $m$-marked cycles shapes and give a concrete criteria for a permutation to belong to the orbit associated to a given $m$-marked cycle shape. \emph{Section 6} uses this product to describe the leading terms of products of basis elements in $\mathsf{FH}_{m}$. In \emph{Section 7} the algebra of symmetric functions is recalled and \emph{Section 8} recalls the Jucys-Murphy elements of $\mathbb{Z}\mathfrak{S}_{n}$ and their relationship to the Farahat-Higman algebra $\mathsf{FH}_{0}$. We end with \emph{Section 9} which proves the isomorphism $\mathsf{FH}_{m}\cong R\otimes_{\mathbb{Z}}(\mathcal{H}_{m}\otimes\mathsf{Sym})$ and summarises various results for $\mathsf{FH}_{m}$ and the centraliser algebras $Z_{n,m}$. 

\begin{ack}
This work was done during the author's PhD at City, University of London which was funded by EPSRC.
\end{ack}

%%%%%%%%%%%%%%%%    POLYNOMIAL CRADINALILTY OF m-CLASSES %%%%%%%%%%%%%%%%%%%%%%%%%%%%%%%%%%%%%%%%%%%%%%

\section{Polynomial Cardinality of $m$-Classes}

For a finite set $A\subset\mathbb{N}$ we write $\mathfrak{S}(A)$ to denote the group of permutations of $A$. For any $n\geq 0$ we write $[n]:=\{1,\dots,n\}$ and $\mathfrak{S}_{n}:=\mathfrak{S}([n])$, with the convention that $[0]=\emptyset$ and $\mathfrak{S}_{0}$ the trivial group. We will set $\mathfrak{S}_{\mathbb{N}}:=\cup_{n\geq 1}\mathfrak{S}_{n}$, the group of permutations of $\mathbb{N}$ with finite support. We denote the support of any $\pi\in\mathfrak{S}_{\mathbb{N}}$ by $\mathsf{Sup}(\pi)=\{i\in\mathbb{N} \ | \ \pi(i)\neq i\}$ and set $||\pi||:=|\mathsf{Sup}(\pi)|$. For any $r\geq 1$ consider the $r$-fold direct product $\mathfrak{S}_{\mathbb{N}}^{\times r}$. For $\bm{\pi}=(\pi_{i})_{i=1}^{r}\in\mathfrak{S}_{\mathbb{N}}^{\times r}$ let $\mathsf{Sup}(\bm{\pi}):=\mathsf{Sup}(\pi_{1})\cup\cdots\cup\mathsf{Sup}(\pi_{r})$ and set $||\bm{\pi}||:=|\mathsf{Sup}(\bm{\pi})|$. For any $m\geq 0$ let $\mathsf{Stab}(m)$ denote the subgroup of $\mathfrak{S}_{\mathbb{N}}$ consisting of the permutations which fix each element of $[m]$, in particular $\mathsf{Stab}(0)=\mathfrak{S}_{\mathbb{N}}$. The group $\mathsf{Stab}(m)$ acts on the $r$-fold direct product $\mathfrak{S}_{\mathbb{N}}^{\times r}$ by component-wise conjugation. We call the respected orbits the $m$-\emph{classes} of $\mathfrak{S}_{\mathbb{N}}^{\times r}$. We denote the $m$-class of $\mathfrak{S}_{\mathbb{N}}^{\times r}$ containing $\bm{\pi}=(\pi_{i})_{i=1}^{r}\in\mathfrak{S}_{\mathbb{N}}^{\times r}$ by
\[ \mathsf{CL}_{\mathbb{N},m}(\bm{\pi}):=\{ (\sigma_{i})_{i=1}^{r}\in\mathfrak{S}_{\mathbb{N}}^{\times r} \ | \ \sigma_{i}=\tau\pi_{i}\tau^{-1} \text{ for all } i\in[r] \text{ and some } \tau\in\mathsf{Stab}(m) \}. \]

\begin{ex} \label{Ex:mClassSN}
Let $m=r=2$ and consider $\bm{\pi}=((1,3)(2,4), (3,4,5))\in\mathfrak{S}_{\mathbb{N}}^{\times 2}$. Then the $m$-class containing $\bm{\pi}$ is given by $\mathsf{CL}_{\mathbb{N},2}(\bm{\pi})=\{ ((1,a)(2,b), (a,b,c)) \ | \ (a,b,c)\in(\mathbb{N}\backslash [2])^{!3} \}$ where $(\mathbb{N}\backslash [2])^{!3}$ is the subset of the $3$-fold direct product of $\mathbb{N}\backslash [2]$ consisting of all tuples with pairwise distinct entries.
\end{ex}

Given any $m\geq 0$ and $\bm{\pi}=(\pi_{i})_{i=1}^{r}\in\mathfrak{S}_{\mathbb{N}}^{\times r}$ define
\[ \mathsf{Sup}_{m}(\bm{\pi}):=\mathsf{Sup}(\bm{\pi})\cap [m], \ \text{and } \  \mathsf{Sup}^{m}(\bm{\pi}):=\mathsf{Sup}(\bm{\pi})\backslash [m], \]
and write $||\bm{\pi}||_{m}$ and $||\bm{\pi}||^{m}$ for the cardinalities of $\mathsf{Sup}_{m}(\bm{\pi})$ and $\mathsf{Sup}^{m}(\bm{\pi})$ respectively. Let $\mathsf{C}$ be an $m$-class of $\mathfrak{S}_{\mathbb{N}}^{\times r}$ and let $\bm{\pi}=(\pi_{i})_{i=1}^{r},\bm{\sigma}=(\sigma_{i})_{i=1}^{r}\in\mathsf{C}$. Hence there exists some $\tau\in\mathsf{Stab}(m)$ such that $\sigma_{i}=\tau\pi_{i}\tau^{-1}$ for each $i\in[r]$. Conjugating $\pi_{i}$ by $\tau$ permutes the entries within the cycles of $\pi_{i}$ according to $\tau$, thus for each $i\in[r]$ the permutations $\pi_{i}$ and $\sigma_{i}$ must have the same cycle structure, and the relative positions of the elements of $[m]$ among their cycles must agree. Hence $||\bm{\pi}||=||\bm{\sigma}||$, $||\bm{\pi}||_{m}=||\bm{\sigma}||_{m}$, and $||\bm{\pi}||^{m}=||\bm{\sigma}||^{m}$, and so it makes sense to define $||\mathsf{C}||:=||\bm{\pi}||$, $||\mathsf{C}||_{m}:=||\bm{\pi}||_{m}$, and $||\mathsf{C}||^{m}:=||\bm{\pi}||^{m}$ for any $\bm{\pi}\in\mathsf{C}$.

Let $n\geq m$, then we set $\mathsf{Stab}_{n}(m):=\mathsf{Stab}(m)\cap\mathfrak{S}_{n}$. Similar to the above situation, this group acts on the $r$-fold direct product $\mathfrak{S}_{n}^{\times r}$ by component-wise conjugation. We call the respected orbits the $m$-\emph{classes} of $\mathfrak{S}_{n}^{\times r}$. We denote the $m$-class of $\mathfrak{S}_{n}^{\times r}$ containing $\bm{\pi}=(\pi_{i})_{i=1}^{r}\in\mathfrak{S}_{n}^{\times r}$ by
\[ \mathsf{CL}_{n,m}(\bm{\pi}):=\{ (\sigma_{i})_{i=1}^{r}\in\mathfrak{S}_{n}^{\times r} \ | \ \sigma_{i}=\tau\pi_{i}\tau^{-1} \text{ for all } i\in[r] \text{ and some } \tau\in\mathsf{Stab}_{n}(m) \}. \]
Again we can set $||\mathsf{C}||:=||\bm{\pi}||$, $||\mathsf{C}||_{m}:=||\bm{\pi}||_{m}$, and $||\mathsf{C}||^{m}:=||\bm{\pi}||^{m}$ for any $m$-class $\mathsf{C}$ of $\mathfrak{S}_{n}^{\times r}$ and any $\bm{\pi}\in\mathsf{C}$. Given any $m$-class $\mathsf{C}$ of $\mathfrak{S}_{\mathbb{N}}^{\times r}$ we write $\mathsf{C}_{n}:=\mathsf{C}\cap\mathfrak{S}_{n}^{\times r}$.

\begin{ex} \label{Ex:mClassSn}
Continuing from \emph{\Cref{Ex:mClassSN}} let $\mathsf{C}=\mathsf{CL}_{\mathbb{N},2}(\bm{\pi})$. For any $n\geq 2$ one can deduce that
\[ \mathsf{C}_{n} = 
\begin{cases}
\{ ((1,a)(2,b), (a,b,c)) \ | \ (a,b,c)\in([n]\backslash [2])^{!3} \}, & n\geq 5\\
\emptyset, & 2\leq n \leq 4 \\
\end{cases}
\]
When $n\geq 5$ the set $\mathsf{C}_{n}$ is an $m$-class of $\mathfrak{S}_{n}^{\times 2}$.
\end{ex}

As the above example suggests, when $\mathsf{C}$ is an $m$-class of $\mathfrak{S}_{\mathbb{N}}^{\times r}$ and $n\geq m$, then the set $\mathsf{C}_{n}$ is either empty or an $m$-class of $\mathfrak{S}_{n}^{\times r}$. The following proposition proves this and gives a criteria for when $\mathsf{C}_{n}$ is empty.

\begin{prop} \label{Prop:mClassInequality}
Let $\mathsf{C}$ be an $m$-class of $\mathfrak{S}_{\mathbb{N}}^{\times r}$ and $n\geq m$. Then $\mathsf{C}_{n}$ is non-empty if and only if
\begin{equation} \label{mClassInequality}
n\geq ||\mathsf{C}||^{m}+m.
\end{equation}
In this case $\mathsf{C}_{n}$ is an $m$-class of $\mathfrak{S}_{n}^{\times r}$ and the $m$-classes of $\mathfrak{S}_{n}^{\times r}$ appear uniquely in this manner.
\end{prop}

\begin{proof}
Fix $\bm{\pi}=(\pi_{i})_{i=1}^{r}\in\mathsf{C}$. All elements in $\mathsf{C}$ are of the form $(\tau\pi_{i}\tau^{-1})_{i=1}^{r}$ for some $\tau\in\mathsf{Stab}(m)$. The set $\mathsf{Sup}^{m}(\tau\bm{\pi}\tau^{-1})$ is all elements of $\mathbb{N}\backslash [m]$ for which at least one $\tau\pi_{i}\tau^{-1}$ acts non-trivially. Thus $\tau\bm{\pi}\tau^{-1}\in\mathsf{C}_{n}$ if and only if $\mathsf{Sup}^{m}(\tau\bm{\pi}\tau^{-1})\subset[n]\backslash [m]$. For this to be the case we must have
\[ |\mathsf{Sup}^{m}(\tau\bm{\pi}\tau^{-1})|=||\mathsf{C}||^{m}\leq |[n]\backslash [m]|=n-m. \]
Rearranging gives \emph{\Cref{mClassInequality}}. So $\mathsf{C}_{n}=\emptyset$ whenever $n<||\mathsf{C}||^{m}+m$. Now assume \emph{\Cref{mClassInequality}} holds. We have $\mathsf{Sup}^{m}(\tau\bm{\pi}\tau^{-1})=\tau\left(\mathsf{Sup}^{m}(\bm{\pi})\right)$, the image of $\mathsf{Sup}^{m}(\bm{\pi})$ under $\tau$. Let $t:\mathsf{Sup}^{m}(\bm{\pi})\rightarrow [n]\backslash [m]$ be an injective map, which exists since \emph{\Cref{mClassInequality}} holds. Then fix a permutation $\tau_{t}\in\mathsf{Stab}(m)$ such that $\tau(i)=t(i)$ for all $i\in\mathsf{Sup}^{m}(\bm{\pi})$. Hence $\mathsf{Sup}^{m}(\tau_{t}\bm{\pi}\tau_{t}^{-1})\subset[n]\backslash [m]$ and so $\mathsf{C}_{n}$ is non-empty as it contains $\bm{\sigma}:=\tau_{t}\bm{\pi}\tau_{t}^{-1}$.

We now want to prove that $\mathsf{C}_{n}$ is an $m$-class of $\mathfrak{S}_{n}^{\times r}$. Since $\bm{\sigma}:=\tau_{t}\bm{\pi}\tau_{t}^{-1}\in\mathsf{C}_{n}$ and $\mathsf{C}_{n}=\mathsf{C}\cap\mathfrak{S}_{n}^{\times r}$, any element of $\mathsf{C}_{n}$ is of the form $\tau\bm{\sigma}\tau^{-1}$ for $\tau\in\mathsf{Stab}(m)$ such that $\mathsf{Sup}^{m}(\tau\bm{\sigma}\tau^{-1})\subset[n]\backslash [m]$. If $\tau\in\mathsf{Stab}_{n}(m)$ then clearly $\mathsf{Sup}^{m}(\tau\bm{\sigma}\tau^{-1})\subset[n]\backslash [m]$, thus we have $\mathsf{CL}_{n,m}(\bm{\sigma})\subset \mathsf{C}_{n}$. Hence the equality $\mathsf{CL}_{n,m}(\bm{\sigma})=\mathsf{C}_{n}$ occurs if one can show that whenever $\tau\bm{\sigma}\tau^{-1}\in\mathsf{C}_{n}$ for some $\tau\in\mathsf{Stab}(m)$ there exists a $\tau'\in\mathsf{Stab}_{n}(m)$ such that $\tau\bm{\sigma}\tau^{-1}=\tau'\bm{\sigma}(\tau')^{-1}$. Given such a $\tau$ we must have that $\mathsf{Sup}^{m}(\tau\bm{\sigma}\tau^{-1})=\tau\left(\mathsf{Sup}^{m}(\bm{\sigma})\right)\subset[n]\backslash [m]$. Since we also have that $\mathsf{Sup}^{m}(\bm{\sigma})\subset[n]\backslash [m]$, let $\tau'$ be any permutation of $[n]\backslash [m]$ with the property that $\tau'(i):=\tau(i)$ for all $i\in\mathsf{Sup}^{m}(\bm{\sigma})$, then it is clear that $\tau\bm{\sigma}\tau^{-1}=\tau'\bm{\sigma}(\tau')^{-1}$ and so $\mathsf{CL}_{n,m}(\bm{\sigma})=\mathsf{C}_{n}$. Hence for any $\bm{\sigma}\in\mathfrak{S}_{n}^{\times r}$ we have $(\mathsf{CL}_{\mathbb{N},m}(\bm{\sigma}))_{n}=\mathsf{CL}_{n,m}(\bm{\sigma})$. Since orbits intersect trivially all such $m$-classes of $\mathfrak{S}_{n}^{\times r}$ appear as the intersection of a unique $m$-class of $\mathfrak{S}_{\mathbb{N}}^{\times r}$ with $\mathfrak{S}_{n}^{\times r}$.

\end{proof}

We now end this section by presenting a result describing the cardinality of the set $\mathsf{C}_{n}$ when $\mathsf{C}$ is an $m$-class of $\mathfrak{S}_{\mathbb{N}}^{\times r}$ and $n\geq m$. Notably we show that the cardinality of such a set is polynomial in $n$.

\begin{prop} \label{Prop:mClassCardinality}
Let $\mathsf{C}$ be an $m$-class of $\mathfrak{S}_{\mathbb{N}}^{\times r}$. For $n\geq m$,
\[ |\mathsf{C}_{n}|= \frac{1}{b(\mathsf{C})}(n-m)(n-m-1)\cdots(n-m-||\mathsf{C}||^{m}+1) \]
where $b(\mathsf{C})\in\mathbb{N}$ is a constant depending only on the class $\mathsf{C}$ and not on $n$. 
\end{prop}

\begin{proof}
Assume $n\geq ||\mathsf{C}||^{m}+m$, so $\mathsf{C}_{n}$ is an $m$-class of $\mathfrak{S}_{n}^{\times r}$ by \emph{\Cref{Prop:mClassInequality}}.  Pick $\bm{\pi}=(\pi_{i})_{i=1}^{r}\in\mathsf{C}_{n}$, then $\mathsf{C}_{n}=\mathsf{CL}_{n,m}(\bm{\pi})$ is the orbit of $\bm{\pi}$ in $\mathfrak{S}_{n}^{\times r}$ under the action of $\mathsf{Stab}_{n}(m)$ by component-wise conjugation. Thus by the \emph{Orbit-Stabilizer Theorem} we have
\[ |\mathsf{C}_{n}|=\frac{|\mathsf{Stab}_{n}(m)|}{|\mathsf{Stab}_{\mathsf{Stab}_{n}(m)}(\bm{\pi})|} \]
where $\mathsf{Stab}_{\mathsf{Stab}_{n}(m)}(\bm{\pi})=\{\tau\in\mathsf{Stab}_{n}(m) \ | \ \tau\bm{\pi}\tau^{-1}=\bm{\pi}\}$ is the subgroup of $\mathsf{Stab}_{n}(m)$ whose elements fix $\bm{\pi}$ under component-wise conjugation. Consider $\mathsf{Stab}_{n}([m]\cup\mathsf{Sup}^{m}(\bm{\pi}))$, the subgroup of $\mathfrak{S}_{n}$ consisting of all permutations which act trivially on $[m]\cup\mathsf{Sup}^{m}(\bm{\pi})$. Naturally $\mathsf{Stab}_{n}([m]\cup\mathsf{Sup}^{m}(\bm{\pi}))\subset\mathsf{Stab}_{\mathsf{Stab}_{n}(m)}(\bm{\pi})$. Now let $\mathfrak{S}(\mathsf{Sup}^{m}(\bm{\pi}))$ denote the subgroup of $\mathfrak{S}_{n}$ consisting of the permutations of $\mathsf{Sup}^{m}(\bm{\pi})$. Then we have that $\mathfrak{S}(\mathsf{Sup}^{m}(\bm{\pi}))\subset\mathsf{Stab}_{n}(m)$, and hence $\mathsf{Stab}_{\mathfrak{S}(\mathsf{Sup}^{m}(\bm{\pi}))}(\bm{\pi})=\{ \tau\in\mathfrak{S}(\mathsf{Sup}^{m}(\bm{\pi})) \ | \  \tau\bm{\pi}\tau^{-1}=\bm{\pi}\}$ is also a subgroup of $\mathsf{Stab}_{\mathsf{Stab}_{n}(m)}(\bm{\pi})$.

\emph{Claim}: We have a group isomorphism
\[ \mathsf{Stab}_{\mathsf{Stab}_{n}(m)}(\bm{\pi})\cong\mathsf{Stab}_{n}([m]\cup\mathsf{Sup}^{m}(\bm{\pi}))\times\mathsf{Stab}_{\mathfrak{S}(\mathsf{Sup}^{m}(\bm{\pi}))}(\bm{\pi}). \]
Note that the two subgroups $\mathsf{Stab}_{n}([m]\cup\mathsf{Sup}^{m}(\bm{\pi}))$ and $\mathsf{Stab}_{\mathfrak{S}(\mathsf{Sup}^{m}(\bm{\pi}))}(\bm{\pi})$ commute and have trivial intersection. So to prove this claim we only need to show that any permutation $\tau\in\mathsf{Stab}_{\mathsf{Stab}_{n}(m)}(\bm{\pi})$ can be expressed as $\tau=\tau_{1}\tau_{2}$ for some $\tau_{1}\in\mathsf{Stab}_{n}([m]\cup\mathsf{Sup}^{m}(\bm{\pi}))$ and $\tau_{2}\in\mathsf{Stab}_{\mathfrak{S}(\mathsf{Sup}^{m}(\bm{\pi}))}(\bm{\pi})$. Let $\tau\in\mathsf{Stab}_{\mathsf{Stab}_{n}(m)}(\bm{\pi})$, then by definition $\tau\pi_{i}\tau^{-1}=\pi_{i}$ for each $i\in[r]$. We seek to show that the sets $[n]\backslash\mathsf{Sup}(\bm{\pi})$ and $\mathsf{Sup}(\bm{\pi})$ are invariant under the action of $\tau$. Suppose for contradiction this is not the case, hence there exists an $a\in [n]\backslash\mathsf{Sup}(\bm{\pi})$ such that $\tau(a)=b\in \mathsf{Sup}(\bm{\pi})$. Then for each $i\in[r]$,
\[ \pi_{i}(b)=(\tau\pi_{i}\tau^{-1})(b) = (\tau\pi_{i})(a) = \tau(a)=b. \]
Thus $\pi_{i}$ fixes $b$ for each $i\in[r]$, but this gives the desired contradiction since $b\in\mathsf{Sup}(\bm{\pi})$. Thus the sets $[n]\backslash\mathsf{Sup}(\bm{\pi})$ and $\mathsf{Sup}(\bm{\pi})$ are invariant under the action of $\tau$, which implies a decomposition $\tau=\tau_{1}\tau_{2}$ where $\tau_{1}$ is a permutation of $[n]\backslash\mathsf{Sup}(\bm{\pi})$ and $\tau_{2}$ is a permutation of $\mathsf{Sup}(\bm{\pi})$. Naturally $\tau_{1}$ and $\tau_{2}$ commute, and since $\tau$ fixes $[m]$, both $\tau_{1}$ and $\tau_{2}$ must also fix $[m]$. As such $\tau_{1}$ is an element of $\mathsf{Stab}_{n}([m]\cup\mathsf{Sup}^{m}(\bm{\pi}))$ as desired, and $\tau_{2}\in\mathfrak{S}(\mathsf{Sup}^{m}(\bm{\pi}))$. Lastly note that for each $i\in[r]$,
\[ \pi_{i}=\tau\pi_{i}\tau^{-1}=\tau_{2}\tau_{1}\pi_{i}\tau_{1}^{-1}\tau_{2}^{-1} = \tau_{2}\pi_{i}\tau_{2}^{-1} \]
since $\tau_{1}$ commutes with both $\tau_{2}$ and $\pi_{i}$. Thus $\tau_{2}\pi_{i}\tau_{2}^{-1}=\pi_{i}$ for each $i\in[r]$, implying that $\tau_{2}$ belongs to $\mathsf{Stab}_{\mathfrak{S}(\mathsf{Sup}^{m}(\bm{\pi}))}(\bm{\pi})$. Hence the claim holds.

Returning to the cardinality of $\mathsf{C}_{n}$, recall that $|\mathsf{Sup}^{m}(\bm{\pi})|=||\mathsf{C}||^{m}$. Also observe that the size of the group $\mathfrak{S}(\mathsf{Sup}^{m}(\bm{\pi}))$ depends only on the class $\mathsf{C}$, in particular it is independent of $n$. This implies the same for the cardinality of $\mathsf{Stab}_{\mathfrak{S}(\mathsf{Sup}^{m}(\bm{\pi}))}(\bm{\pi})$, and so we write $b(\mathsf{C}):=|\mathsf{Stab}_{\mathfrak{S}(\mathsf{Sup}^{X}(\bm{\pi}))}(\bm{\pi})|$. Hence we have that
\begin{align*}
|\mathsf{C}_{n}|&=\frac{|\mathsf{Stab}_{n}(m)|}{|\mathsf{Stab}_{n}([m]\cup\mathsf{Sup}^{m}(\bm{\pi}))||\mathsf{Stab}_{\mathfrak{S}(\mathsf{Sup}^{m}(\bm{\pi}))}(\bm{\pi})|} \\
&= \frac{(n-m)!}{(n-m-||\mathsf{C}||^{m})!b(\mathsf{C})} \\
&= \frac{1}{b(\mathsf{C})}(n-m)(n-m-1)\cdots(n-m-||\mathsf{C}||^{m}+1).
\end{align*}
We assumed $n\geq ||\mathsf{C}||^{m}+m$ so that $\mathsf{C}_{n}$ is an $m$-class of $\mathfrak{S}_{n}^{\times r}$. By \emph{\Cref{Prop:mClassInequality}} the set $\mathsf{C}_{n}$ is empty when $m\leq n <||\mathsf{C}||^{m}+m$, and such values of $n$ are precisely those which give zero in the above formula for $|\mathsf{C}_{n}|$. Hence this formula holds for all $n\geq m$, completing the proof.

\end{proof}

%%%%%%%%%%%%%%%%    CENTRALISER ALGEBRAS   %%%%%%%%%%%%%%%%%%%%%%%%%%%%%%%%%%%%%%%%%%%%%%%%%%%%%%

\section{Centraliser Algebras and Polynomial Structure Constants}

Let $m\geq 0$, $\mathsf{C}$ be an $m$-class of $\mathfrak{S}_{\mathbb{N}}$, and $\pi\in\mathsf{C}$. As mentioned before $\mathsf{C}$ is completely determined by the cycle structure of $\pi$ and the relative positions of the elements of $[m]$ among the cycles. If one was to take $\pi$ and replace each of the elements in $\mathbb{N}\backslash [m]$ among the non-trivial cycles with a formal symbol, say $*$, then the resulting object would retain the defining characteristics of $\mathsf{C}$, and so could represent it.

\begin{defn} \label{Defn:mMCS}
For a finite set $A$, we call $(a_{i})_{i=1}^{l}\in A^{\times l}$, for some $l\in\mathbb{N}$, a \emph{cycle} if we only care about the order of the entries up to cyclic shifts, and say it has length $l$. For $*$ a formal symbol, we define an $m$-\emph{marked cycle shape} to be a finite collection of cycles with entries in $[m]\cup\{*\}$ with the following properties:
\begin{itemize}
\item[(1)] The multiset of entries among the cycles equals $[m]\cup\{*^{n}\}$ for some $n\in\mathbb{Z}_{\geq 0}$, in particular each element of $[m]$ appears precisely once.
\item[(2)] Cycles containing only $*$ must be of length at least two. 
\end{itemize}
We write an $m$-marked cycle shape as a formal product of its cycles by juxtaposition, where the order of the cycles is immaterial. We let $\Lambda(m)$ denote the set of all such $m$-marked cycle shapes.
\end{defn}

\begin{ex} \label{Ex:mMCS}
An example of a $6$-marked cycle shape is
\[ \lambda = (2,6)(5)(1,*,*,4,*)(3,*,*)(*,*)(*,*,*)\in\Lambda(6). \]
The multiset of entries among the cycles of $\lambda$ is $[6]\cup\{*^{10}\}$. Cyclicly shifting the entries of any of the cycles or rearranging the cycles in any order will result in an alterntaive expression of $\lambda$.
\end{ex}

The $0$-marked cycle shapes only contain the symbol $*$, and we will refer to these as \emph{cycle shapes}.  Clearly one may identify the set $\Lambda(0)$ with the set of proper partitions $\mathcal{P}^{\text{pr}}$. We consider the empty set $\emptyset$ an element of $\Lambda(0)$ consisting of no cycles. The subset of $\Lambda(m)$ consisting of the $m$-marked cycle shapes which contain no symbols $*$ may be identified with $\mathfrak{S}_{m}$, and so we write $\mathfrak{S}_{m}\subset\Lambda(m)$. 

As discussed above the set of $m$-marked cycles shapes $\Lambda(m)$ provides a natural indexing set for the collection of $m$-classes of $\mathfrak{S}_{\mathbb{N}}$: For $\lambda\in\Lambda(m)$ let $\mathsf{CL}_{\mathbb{N},m}(\lambda)$ denote the corresponding $m$-class. The permutations of $\mathsf{CL}_{\mathbb{N},m}(\lambda)$ are those obtained from $\lambda$ by replacing the symbols $*$ with distinct elements from $\mathbb{N}\backslash [m]$. For example, letting $\lambda$ be the $6$-marked cycle shape displayed in \emph{\Cref{Ex:mMCS}}, we have that
\[ \mathsf{CL}_{\mathbb{N},6}(\lambda)=\left\{ (2,6)(1,a_{1},a_{2},4,a_{3})(3,a_{4},a_{5})(a_{6},a_{7})(a_{8},a_{9},a_{10}) \ \Big| \ (a_{i})_{i=1}^{10}\in(\mathbb{N}\backslash [6])^{!10} \right\}. \]
Recall the quantities $||\mathsf{C}||$, $||\mathsf{C}||^{m}$, and $||\mathsf{C}||_{m}$ for any $m$-class $\mathsf{C}$ of $\mathfrak{S}_{\mathbb{N}}$. Then for $\lambda\in\Lambda(m)$ we write
\[ ||\lambda||:=||\mathsf{CL}_{\mathbb{N},m}(\lambda)||, \ \ ||\lambda||^{m}:=||\mathsf{CL}_{\mathbb{N},m}(\lambda)||^{m}, \ \text{ and } \ ||\lambda||_{m}:=||\mathsf{CL}_{\mathbb{N},m}(\lambda)||_{m}. \]
The quantity $||\lambda||$ is the number of entries among the cycles of length at least two, $||\lambda||^{m}$ is the number of $*$ symbols appearing among the cycles, and $||\lambda||_{m}$ is the number of elements from $[m]$ which appear in cycles of length at least two. We define a map $\mathsf{deg}_{m}:\Lambda(m)\rightarrow\mathbb{Z}_{\geq m}$ by setting
\[ \mathsf{deg}_{m}(\lambda):=||\lambda||^{m}+m. \]
We call $\mathsf{deg}_{m}(\lambda)$ the \emph{degree} of $\lambda$, which is defined with the inequality of \emph{\Cref{Prop:mClassInequality}} in mind. Specially, for $n\geq m$, the $m$-classes of $\mathfrak{S}_{n}$ are indexed by $\Lambda_{\leq n}(m):=\{\lambda\in\Lambda(m) \ | \ \mathsf{deg}_{m}(\lambda)\leq n\}$. For $\lambda\in\Lambda_{\leq n}(m)$ we denote the corresponding $m$-class by $\mathsf{CL}_{n,m}(\lambda)=(\mathsf{CL}_{\mathbb{N},m}(\lambda))_{n}$, which consists of all permutations of $\mathfrak{S}_{n}$ one can obtain from $\lambda$ by replacing the symbols $*$ with distinct elements from $[n]\backslash [m]$.

Consider the $\mathbb{Z}$-algebra $\mathbb{Z}\mathfrak{S}_{n}$. For any $n\geq m$ we define the $m$-\emph{centraliser algebra} as
\[ Z_{n,m}:=\{z\in\mathbb{Z}\mathfrak{S}_{n} \ | \ \tau z=z\tau \text{ for all } \tau\in\mathsf{Stab}_{n}(m)\}. \]
When $m=0$ then $\mathsf{Stab}_{n}(0)=\mathfrak{S}_{n}$ and $Z_{n,0}=Z(\mathbb{Z}\mathfrak{S}_{n})$. For any $\lambda\in\Lambda(m)$ we define the $m$-class sum by
\[ K_{n}(\lambda):=\sum_{\pi\in\mathsf{CL}_{n,m}(\lambda)}\pi. \]
Note by \emph{\Cref{Prop:mClassInequality}} we have that $K_{n}(\lambda)\neq 0$ if and only if $\mathsf{deg}_{m}(\lambda)\leq n$. Also one can deduce that the set of $m$-class sums $\{K_{n}(\lambda) \ | \ \lambda\in\Lambda_{\leq n}(m)\}$ provides a $\mathbb{Z}$-basis for the centraliser algebra $Z_{n,m}$. As such the product of any two such elements must decomposed into a linear combination of the same, that is for any $\lambda,\mu\in\Lambda_{\leq n}(m)$ we have that
\[ K_{n}(\lambda)K_{n}(\mu)=\sum_{\nu\in\Lambda_{\leq n}(m)}a_{\lambda,\mu}^{\nu}(n)K_{n}(\nu) \]
with structure constants $a_{\lambda,\mu}^{\nu}(n)\in\mathbb{Z}_{\geq 0}$. We now show these structure constants are polynomial in $n$. Let $t$ be a formal variable and $R:=\mathsf{Int}[t]$ the subring of $\mathbb{Q}[t]$ of all polynomials $p(t)$ such that $p(\mathbb{Z})\subseteq\mathbb{Z}$.

\begin{thm} \label{Thm:PolyStructureConstants}
For each $\lambda,\mu,\nu\in\Lambda(m)$ there exists a unique polynomial $f_{\lambda,\mu}^{\nu}(t)\in R$ such that
\[ K_{n}(\lambda)K_{n}(\mu) = \sum_{\nu\in\Lambda_{\leq n}(m)}f_{\lambda,\mu}^{\nu}(n)K_{n}(\nu) \]
in $Z_{n,m}$ for any $n\geq m$. We refer to the polynomials $f_{\lambda,\mu}^{\nu}(t)$ as the \emph{structure polynomials}.
\end{thm}

\begin{proof}
Fix $\lambda,\mu,\nu\in\Lambda(m)$. Consider the set of pairs
\[ A = \{ (\pi_{1},\pi_{2})\in\mathsf{CL}_{\mathbb{N},m}(\lambda)\times\mathsf{CL}_{\mathbb{N},m}(\mu) \ | \ \pi_{1}\pi_{2}\in\mathsf{CL}_{\mathbb{N},m}(\nu) \} \subset \mathfrak{S}_{\mathbb{N}}\times\mathfrak{S}_{\mathbb{N}}. \]
When $A=\emptyset$ we set $f_{\lambda,\mu}^{\nu}(t):=0$. Assume $A\neq\emptyset$ and let $(\pi_{1},\pi_{2})\in A$. For any $\tau\in\mathsf{Stab}(m)$ we have that $(\tau\pi_{1}\tau^{-1})(\tau\pi_{2}\tau^{-1})=\tau(\pi_{1}\pi_{2})\tau^{-1}$ which belongs to $\mathsf{CL}_{\mathbb{N},m}(\nu)$ since $\pi_{1}\pi_{2}$ does. So $(\tau\pi_{1}\tau^{-1}, \tau\pi_{2}\tau^{-1})$ belongs to $A$ for any $\tau\in\mathsf{Stab}(m)$. Thus for some indexing set $I$, the set $A$ is the union of $m$-class $\mathsf{C}^{(i)}$ of $\mathfrak{S}_{\mathbb{N}}\times\mathfrak{S}_{\mathbb{N}}$ for each $i\in I$. For any $(\pi_{1}^{(i)},\pi_{2}^{(i)})\in\mathsf{C}^{(i)}$,
\[ ||\mathsf{C}^{(i)}||^{m} = |\mathsf{Sup}^{m}(\pi_{1}^{(i)})\cup\mathsf{Sup}^{m}(\pi_{2}^{(i)})| \leq |\mathsf{Sup}^{m}(\pi_{1}^{(i)})|+|\mathsf{Sup}^{m}(\pi_{2}^{(i)})|=||\lambda||^{m}+||\mu||^{m}. \]
Thus by \emph{\Cref{Prop:mClassInequality}} we have for any $n\geq ||\lambda||^{m}+||\mu||^{m}+m$ that $\mathsf{C}_{n}^{(i)}$ is an $m$-class of $\mathfrak{S}_{n}\times\mathfrak{S}_{n}$ for each $i\in I$. This implies that $I$ is finite. Also by \emph{\Cref{Prop:mClassCardinality}} we have for any $n\geq m$ that
\[ |A\cap (\mathfrak{S}_{n}\times\mathfrak{S}_{n})| = \sum_{i\in I}\frac{1}{b(\mathsf{C}^{(i)})}(n-m)(n-m-1)\cdots(n-m-||\mathsf{C}^{(i)}||^{m}+1), \]
where $b(\mathsf{C}^{(i)})$ are constants independent of $n$. By definition of $A$, the multiplicity of $K_{n}(\nu)$ in the product $K_{n}(\lambda)K_{n}(\mu)$ is $|A\cap (\mathfrak{S}_{n}\times\mathfrak{S}_{n})|$ divided by $|\mathsf{CL}_{n,m}(\nu)|$. Again by \emph{\Cref{Prop:mClassCardinality}} we have that
\[ |\mathsf{CL}_{n,m}(\nu)| = \frac{1}{b(\nu)}(n-m)(n-m-1)\cdots(n-m-||\nu||^{m}+1) \]
where $b(\nu)=b(\mathsf{CL}_{\mathbb{N},m}(\nu))$ is a constant independent of $n$. Now for any $(\pi_{1}^{(i)},\pi_{2}^{(i)})\in\mathsf{C}^{(i)}$,
\[ ||\nu||^{m} = |\mathsf{Sup}^{m}(\pi_{1}^{(i)}\pi_{2}^{(i)})|\leq |\mathsf{Sup}^{m}(\pi_{1}^{(i)})\cup\mathsf{Sup}^{m}(\pi_{2}^{(i)})|=||\mathsf{C}^{(i)}||. \]
Thus we have that $|A\cap (\mathfrak{S}_{n}\times\mathfrak{S}_{n})|$ divided by $|\mathsf{CL}_{n,m}(\nu)|$ is given by
\[ b(\nu)\sum_{i\in I}\frac{1}{b(\mathsf{C}^{(i)})}(n-m-||\nu||^{m})(n-m-||\nu||^{m}-1)\cdots(n-m-||\mathsf{C}^{(i)}||^{m}+1). \]
Hence let $f_{\lambda,\mu}^{\nu}(t)$ be the polynomial obtained from the above expression by replacing $n$ with $t$. What remains to be shown is that $f_{\lambda,\mu}^{\nu}(t)$ belongs to $R$. It is known that a polynomial $f(t)\in\mathbb{Q}[t]$ of degree $d$ belongs to $R$ if and only if it is integer valued on $d+1$ consecutive integers. Well we have shown that each $f_{\lambda,\mu}^{\nu}(t)$ is integer valued on the infinite set $\{m,m+1,\dots\}$ since the evaluation at such an integer $n$ gives the coefficient of the term $K_{n}(\nu)$ in the product $K_{n}(\lambda)K_{n}(\mu)$ in the $\mathbb{Z}$-algebra $Z_{n,m}$, hence $f_{\lambda,\mu}^{\nu}(t)\in R$.  

\end{proof}

Specialising the above theorem to the case $m=0$ recovers \cite[Theorem 2.2]{FH59}. It is worth remarking that in \cite{FH59} they instead used \emph{reduced partitions} as the indexing set for their classes as aposed to cycle shapes, but there is a natural correspondence between them.

%%%%%%%%%%%%%%%%    FH ALGEBRAS  %%%%%%%%%%%%%%%%%%%%%%%%%%%%%%%%%%%%%%%%%%%%%%%%%%%%%%%%%%%

\section{The Farahat-Higman Algebras $\mathsf{FH}_{m}$}

Knowing that the structure constants of $Z_{n,m}$ are polynomial in $n$ allows us to define a new $R$-algebra $\mathsf{FH}_{m}$, in an analogous manner to the Farahat-Higman algebra. We first define $\mathsf{FH}_{m}$ as a free $R$-module equipped with a product, and then shortly prove that such a product realises it as an $R$-algebra. 

\begin{defn} \label{Defn:FHAlgebra}
Let $\mathsf{FH}_{m}$ be the free $R$-module with formal basis $\{K(\lambda) \ | \ \lambda\in\Lambda(m)\}$. Equip this module with the product given by the $R$-linear extension of
\[ K(\lambda)K(\mu)=\sum_{\nu\in\Lambda(m)}f_{\lambda,\mu}^{\nu}(t)K(\nu) \]
where $f_{\lambda,\mu}^{\nu}(t)$ are the structure polynomials given in \emph{\Cref{Thm:PolyStructureConstants}}.
\end{defn}

For $n\geq m$, $f_{\lambda,\mu}^{\nu}(n)$ gives the multiplicity of $K_{n}(\nu)$ in the product $K_{n}(\lambda)K_{n}(\mu)$ in $Z_{n,m}$. Hence $f_{\lambda,\mu}^{\nu}(t)=0$ whenever $||\nu||>||\lambda||+||\mu||$. As such the product described above for $\mathsf{FH}_{m}$ is well-defined since only a finite number of terms will appear in the product of any two elements.

We say a \emph{distributive ring} is an object satisfying all the axioms of a ring except possibly the associativity of the product and the existence of a multiplicative identity. Thus a ring is a special case of a distributive ring. From definition we certainly have that $\mathsf{FH}_{m}$ is a distributive ring. We seek to show that the product given in \emph{\Cref{Defn:FHAlgebra}} is associative and that $\mathsf{FH}_{m}$ possess a multiplicative identity. This will follow since $\mathsf{FH}_{m}$ is determined by the algebras $Z_{n,m}$ for all $n\geq m$, as we will now demonstrate. 

By definition of the structrure polynomials we must have a surjective homomorphism of distributive rings $\mathsf{pr}_{n,m}:\mathsf{FH}_{m}\rightarrow Z_{n,m}$ given by $\mathsf{pr}_{n,m}(K(\lambda))=K_{n}(\lambda)$ and $\mathsf{pr}_{n,m}(f(t))=f(n)$ for all $\lambda\in\Lambda(m)$ and $f(t)\in R$.

\begin{lem} \label{Lem:TrivialKerIntersection}
The intersection of the kernels of $\mathsf{pr}_{n,m}$ for all $n\geq m$ is trivial, that is
\[ K:=\bigcap_{n\geq m}\mathsf{Ker}(\mathsf{pr}_{n,m})=\{0\} \]
\end{lem}

\begin{proof}
Let $L$ be any finite subset of $\Lambda(m)$, and
\[ x:=\sum_{\lambda\in L}p_{\lambda}(t)K(\lambda). \]
Assume that $x\in K$, we seek to show that $x=0$. Let $n>\text{max}\{\mathsf{deg}_{m}(\lambda) \ | \ \lambda\in L\}$, then
\[ \mathsf{pr}_{n,m}(x)=\sum_{\lambda\in L}p_{\lambda}(n)K_{n}(\lambda)=0, \] 
since $x\in\mathsf{Ker}(\mathsf{pr}_{n,m})$. However, the set $\{K_{n}(\lambda) \ | \ \lambda\in\Lambda(m), \ \mathsf{deg}_{m}(\lambda)\leq n\}$ forms a $\mathbb{Z}$-basis for $Z_{n,m}$, so for the above equation to hold we require that $n$ is a root of $p_{\lambda}(t)$ for each $\lambda\in L$. But this must hold for infinitely many choices of $n$, implying that each polynomial $p_{\lambda}(t)$ must be zero, and so $x=0$ as desired.

\end{proof}

\begin{lem} \label{Lem:FHEqualityByProjection}
For any $X,Y\in\mathsf{FH}_{m}$, then $X=Y$ if and only if $\mathsf{pr}_{n,m}(X)=\mathsf{pr}_{n,m}(Y)$ for all $n\geq m$.
\end{lem} 

\begin{proof}
The forward implication is immediate, while the reverse implication follows since it implies that $X-Y$ belongs to $\cap_{n\geq m}\mathsf{Ker}(\mathsf{pr}_{n,m})$ which equals $\{0\}$ from \emph{\Cref{Lem:TrivialKerIntersection}}.

\end{proof}

Thus we can solve any product in $\mathsf{FH}_{m}$ by computing a corresponding one in $Z_{n,m}$ for arbitrary $n\geq m$.

\begin{ex}
Consider the $2$-marked cycle shapes $\lambda=(1,2)(*,*)$ and $\mu=(1)(2)(*,*)$. We have that $\mathsf{deg}_{2}(\lambda)=\mathsf{deg}_{2}(\mu)=4$, hence $K_{n}(\lambda)$ and $K_{n}(\mu)$ are non-zero if and only if $n\geq4$. Let $n\geq 2$, then
\[ K_{n}(\lambda)K_{n}(\mu)=\left(\sum_{\{a,b\}\subseteq[n]\backslash [2]}(1,2)(a,b)\right)\left(\sum_{\{c,d\}\subseteq[n]\backslash [2]}(c,d)\right) =\sum_{\{a,b\}\subseteq[n]\backslash [2]}\sum_{\{c,d\}\subseteq[n]\backslash [2]}(1,2)(a,b)(c,d). \]
If the two cycles $(a,b)$ and $(c,d)$ are disjoint then the resulting permutation is $(a,b)(c,d)$, and there are two such ways to obtain such. If the two cycles share a single element, then the resulting permutation gives a 3-cycle $(a,b,c)$, and the number of ways to arrive at this 3-cycle from a product of two transpositions is three. Lastly if the two cycles agree then the result is the identity, and there are as many ways to do this as there are two-element subsets of $[n]\backslash [2]$. Thus altogether we have that
\[ K_{n}(\lambda)K_{n}(\mu)=2K_{n}(\tau_{1})+3K_{n}(\tau_{2})+\binom{n-2}{2}K_{n}(\tau_{3}) \]
where $\tau_{1}=(1,2)(*,*)(*,*)$, $\tau_{2}=(1,2)(*,*,*)$, and $\tau_{3}=(1,2)$. Note that
\[ \binom{n-2}{2}=\frac{1}{2}(n-2)(n-3). \]
Thus $K_{n}(\lambda)K_{n}(\mu)=2K_{n}(\tau_{1})+3K_{n}(\tau_{2})+\frac{1}{2}(n-2)(n-3)K_{n}(\tau_{3})$ for all $n\geq 2$, noting that when $n=2,3$ both sides of the equality are zero since $K_{n}(\tau)=0$ for all $\tau\in\{\lambda,\mu,\tau_{1},\tau_{2}\}$, and the polynomial in $n$ which is the coefficient of $K_{n}(\tau_{3})$ has both $2$ and $3$ as roots. Thus by \emph{\Cref{Lem:FHEqualityByProjection}} we have that the relation
\[ K(\lambda)K(\mu)=2K(\tau_{1})+3K(\tau_{2})+\frac{1}{2}(z-2)(z-3)K(\tau_{3}) \]
holds in $\mathsf{FH}_{m}$, in particular $f_{\lambda,\mu}^{\tau_{1}}(z)=2, f_{\lambda,\mu}^{\tau_{2}}(z)=3$, and $f_{\lambda,\mu}^{\tau_{3}}(z)=\frac{1}{2}(z-2)(z-3)$.
\end{ex}

\begin{prop} \label{Prop:FHisRAlg}
The distributive ring $\mathsf{FH}_{m}$ is an $R$-algebra.
\end{prop}

\begin{proof}
We need to prove that the product is associative and a multiplicative identity exists. For the identity, consider $K(\emptyset_{m})$ where $\emptyset_{m}$ is the $m$-marked cycle shape $(1)(2)\dots (m)$. The $m$-class $\mathsf{CL}_{n,m}(\emptyset_{m})$ contains only the identity, so $K_{n}(\emptyset_{m})$ is the identity of $Z_{n,m}$ for any $n\geq m$, implying $K(\emptyset_{m})$ is the identity of $\mathsf{FH}_{m}$ by \emph{\Cref{Lem:FHEqualityByProjection}}. For any $(X,Y,W)\in Z_{n,m}^{\times 3}$ let $[X,Y,W]:=(XY)W-X(YW)$. Then $\mathsf{pr}_{n,m}([X,Y,W])=0$ for all $n\geq m$ as $Z_{n,m}$ is associative. So $[X,Y,W]=0$ by \emph{\Cref{Lem:FHEqualityByProjection}}, showing that $\mathsf{FH}_{m}$ is associative.

\end{proof}

We end this section by describing two natural $R$-subalgebras of $\mathsf{FH}_{m}$.

\begin{lem} \label{Lem:SmSubAlgebra}
We have an injective $R$-algebra homomorphism $R\mathfrak{S}_{m}\rightarrow\mathsf{FH}_{m}$ defined by the $R$-linear extension of $\pi\mapsto K(\pi)$ for all $\pi\in\mathfrak{S}_{m} (\subset\Lambda(m))$.
\end{lem}

\begin{proof}
For $n\geq m$, $\mathsf{CL}_{n,m}(\pi)=\{\pi\}$ for any $\pi\in\mathfrak{S}_{m}$. Hence $\mathsf{pr}_{n,m}(K(\pi))=K_{n}(\pi)=\pi$ showing we have a homomorphism by \emph{\Cref{Lem:FHEqualityByProjection}}. Injectivity follows by construction of $\mathsf{FH}_{m}$.

\end{proof}

Consider the injective map $(-)_{m}:\Lambda(0)\rightarrow\Lambda(m)$ which sends a cycle shape $\lambda\in\Lambda(0)$ to the $m$-marked cycle shape $\lambda_{m}$ obtained from $\lambda$ by adjoining the trivial cycles $(1)(2)\dots (m)$.

\begin{lem} \label{Lem:FH0SubAlgebra}
We have an injective $R$-algebra homomorphism $\mathsf{FH}_{0}\rightarrow\mathsf{FH}_{m}$ defined by the $R$-linear extension of $K(\lambda)\mapsto K(\lambda_{m})$ for all $\lambda\in\Lambda(0)$.
\end{lem}

\begin{proof}
By \emph{\Cref{Lem:FHEqualityByProjection}} one can show that $f_{\lambda_{m},\mu_{m}}^{\nu}(t)=0$ for all $\lambda,\mu\in\Lambda(0)$ whenever $\nu\neq\gamma_{m}$ for any $\gamma\in\Lambda(0)$. Thus it suffices to show that $f_{\lambda_{m},\mu_{m}}^{\nu_{m}}(t)=f_{\lambda,\mu}^{\nu}(t)$ for any $\lambda,\mu,\nu\in\Lambda(0)$. Let $\mathsf{FH}_{m}^{*}$ denote the $R$-subalgebra of $\mathsf{FH}_{m}$ generated by $K(\lambda_{m})$ for all $\lambda\in\Lambda(0)$. Clearly for any $n\geq 0$ we have that $\mathsf{pr}_{n+m,m}(\mathsf{FH}_{m}^{*})=Z(\mathbb{Z}\mathfrak{S}([n+m]\backslash [m]))$. Similarly for any $n\geq 0$ we have that $\mathsf{pr}_{n,0}(\mathsf{FH}_{0})=Z(\mathbb{Z}\mathfrak{S}_{n})$. Clearly we have a $\mathbb{Z}$-algebra isomorphism $Z(\mathbb{Z}\mathfrak{S}([n+m]\backslash [m])) \cong Z(\mathbb{Z}\mathfrak{S}_{n})$, which may be realised by the $R$-linear extension of conjugation $\sigma_{n}(-)\sigma_{n}^{-1}:Z(\mathbb{Z}\mathfrak{S}([n+m]\backslash [m])) \rightarrow Z(\mathbb{Z}\mathfrak{S}_{n})$ where $\sigma_{n}$ is any permutation of $\mathfrak{S}_{n+m}$ such that $\sigma([n+m]\backslash [m])=[n]$. By definition of $(-)_{m}$ we have that $K_{n}(\lambda)=\sigma_{n}K_{n+m}(\lambda_{m})\sigma_{n}^{-1}$ for all $n\geq 0$ and $\lambda\in\Lambda(0)$. Hence by applying \emph{\Cref{Lem:FHEqualityByProjection}} we deduce that $f_{\lambda_{m},\mu_{m}}^{\nu_{m}}(t)=f_{\lambda,\mu}^{\nu}(t)$ for any $\lambda,\mu,\nu\in\Lambda(0)$. Injectivity of $\mathsf{FH}_{0}\rightarrow\mathsf{FH}_{m}$ follows since $\{K(\lambda_{m}) \ | \ \lambda\in\Lambda(0)\}$ is $R$-linearly independent by construction.

\end{proof}

With this result in mind we will often identify $\mathsf{FH}_{0}$ as the $R$-subalgebra of $\mathsf{FH}_{m}$ consisting of $R$-linear combination of basis elements $K(\lambda)$ where $\lambda$ contains the trivial cycles $(1)(2)\cdots (m)$, i.e. indetifying $\mathsf{FH}_{0}$ with the $R$-subalgebra $\mathsf{FH}_{m}^{*}$ described in the above proof. Since $\mathsf{pr}_{n,0}(\mathsf{FH}_{0})=Z(\mathbb{Z}\mathfrak{S}_{n})$, it is easy to deduce from \emph{\Cref{Lem:FHEqualityByProjection}} that $\mathsf{FH}_{0}$ is commutative, and hence a commutative $R$-subalgebra of $\mathsf{FH}_{m}$.

%%%%%%%%%%%%%%%%    MARKED CYCLE SHAPE MONOID %%%%%%%%%%%%%%%%%%%%%%%%%%%%%%%%%%%%%%%%%%%%%%%%%%%%

\section{Monoid of $m$-Marked Cycle Shapes}

To show certain structural properties of the algebras $\mathsf{FH}_{m}$ it will prove helpful to equip the set of $m$-marked cycle shapes $\Lambda(m)$ with an associative product. To do such we will construct a monoid which we may identify with $\Lambda(m)$. This will also allow us to give a more concrete criteria for a permutation to belong to an $m$-class.

Let $\mathcal{U}_{m}$ be the free commutative monoid on the set $\{u_{1},\dots,u_{m}\}$. We have a natural left monoid action $\varphi:\mathfrak{S}_{m}\rightarrow\mathsf{End}(\mathcal{U}_{m})$ given by $\varphi(\pi)(u_{i})=u_{\pi(i)}$ for all $\pi\in\mathfrak{S}_{m}$ and $i\in[m]$, where $\mathsf{End}(\mathcal{U}_{m})$ is the monoid of all monoid endomorphisms $\mathcal{U}_{m}\rightarrow\mathcal{U}_{m}$. Now let $\mathcal{C}$ denote the free commutative monoid on the infinite set $\{c_{1},c_{2},\dots\}$. Then the monoid we are interested in is $\left(\mathfrak{S}_{m}\ltimes_{\varphi}\mathcal{U}_{m}\right)\times\mathcal{C}$, where $\ltimes_{\varphi}$ denotes the semidirect product with respect to the action $\varphi$. The underlying set is $\mathfrak{S}_{m}\times\mathcal{U}_{m}\times\mathcal{C}$ and the product is given by
\[ (\pi,p,a)(\sigma,q,b)=(\pi\sigma,\varphi(\sigma)(p)q,ab) \]
for $(\pi,p,a),(\sigma,q,b)\in\mathfrak{S}_{m}\times\mathcal{U}_{m}\times\mathcal{C}$. We abuse notation and write $\pi=(\pi,1,1)$, $p=(1,p,1)$, and $a=(1,1,a)$. For any set $A$ we let $\mathbb{Z}_{\geq 0}^{A}$ denote the set of all functions $\bm{f}:A\rightarrow\mathbb{Z}_{\geq 0}$ with finite support, that is the subset of elements $a\in A$ such that $\bm{f}(a)\neq 0$ is finite. Then we define
\[ u^{\bm{d}}:=\prod_{i=1}^{m}u_{i}^{\bm{d}(i)} \ \text{ and } \ c^{\bm{l}}:=\prod_{i\in\mathbb{N}}c_{i}^{\bm{l}(i)} \]
for any $\bm{d}\in\mathbb{Z}_{\geq 0}^{[m]}$ and $\bm{l}\in\mathbb{Z}_{\geq 0}^{\mathbb{N}}$, which are well-defined by commutativity and since $\bm{l}$ has finite support. Then as sets we have that
\[ \left(\mathfrak{S}_{m}\ltimes_{\varphi}\mathcal{U}_{m}\right)\times\mathcal{C} = \left\{ \pi u^{\bm{d}}c^{\bm{l}} \ | \ \pi\in\mathfrak{S}_{m}, \bm{d}\in\mathbb{Z}_{\geq 0}^{[m]}, \bm{l}\in\mathbb{Z}_{\geq 0}^{\mathbb{N}} \right\}. \]
Moreover, the product may be described by
\begin{equation} \label{mMarkedCycleShapeProduct}
(\pi u^{\bm{d}}c^{\bm{l}})(\sigma u^{\bm{e}}c^{\bm{k}})=\pi\sigma u^{\sigma\circ\bm{d}+\bm{e}}c^{\bm{l}+\bm{k}}
\end{equation}
where $\sigma\circ\bm{d}\in\mathbb{Z}_{\geq 0}^{[m]}$ is defined by $(\sigma\circ\bm{d})(i)=\bm{d}(\sigma^{-1}(i))$. In particular, the operator $\circ$ realises the set $\mathbb{Z}_{\geq 0}^{[m]}$ as a left action set of $\mathfrak{S}_{m}$. The set $\Lambda(m)$ of $m$-marked cycle shapes is in a natural one-to-one correspondence with this monoid: Consider the map $\left(\mathfrak{S}_{m}\ltimes_{\varphi}\mathcal{U}_{m}\right)\times\mathcal{C}\rightarrow\Lambda(m)$ given by sending $\pi u^{\bm{d}}c^{\bm{l}}$ to the $m$-marked cycle shape consisting, for each $i\in\mathbb{N}$, $\bm{l}(i)$ many cycles of length $i+1$ containing only the symbol $*$, and where the remaining cycles are constructed from those of $\pi$ by adding $\bm{d}(i)$ symbols $*$ after the entry $i$, for each $i\in[m]$ (see examples below). This map is a bijection since there is a natural inverse to consider. With this correspondence in mind, we identify $\Lambda(m)$ with the monoid $\left(\mathfrak{S}_{m}\ltimes_{\varphi}\mathcal{U}_{m}\right)\times\mathcal{C}$.

\begin{exs}
\begin{itemize} \
\item[(1)] Let $\lambda$ be the $6$-marked cycle shape given in \emph{\Cref{Ex:mMCS}}, then we have the identification
\[ (1,4)(2,6)(3)(5)u_{1}^{\textcolor{green}{2}}u_{3}^{\textcolor{cyan}{2}}u_{4}^{\textcolor{violet}{1}}c_{1}^{\textcolor{red}{1}}c_{2}^{\textcolor{blue}{1}} = (2,6)(5)(1,\textcolor{green}{*},\textcolor{green}{*},4,\textcolor{violet}{*})(3,\textcolor{cyan}{*},\textcolor{cyan}{*})\textcolor{red}{(*,*)}\textcolor{blue}{(*,*,*)}, \]
where we have added colours to aid in demonstrating the correspondence.
\item[(2)] Consider the $4$-marked cycle shapes written as elements of $\left(\mathfrak{S}_{m}\ltimes_{\varphi}\mathcal{U}_{m}\right)\times\mathcal{C}$ by $\lambda=(1,2)(3,4)u_{1}u_{2}$, and $\mu=(1,4)(2)(3)u_{3}^{2}c_{1}$. Their product is given by
\[ \lambda\mu=\left((1,2)(3,4)u_{1}u_{2}\right)\left((1,4)(2)(3)u_{3}^{2}c_{1}\right)=(1,3,4,2)u_{2}u_{3}^{2}u_{4}c_{1}. \]
Thus when writting $\lambda$ and $\mu$ as $m$-marked cycle shapes we have
\[ \lambda\mu=\left((1,*,2,*)(3,4)\right)\left((1,4)(2)(3,*,*)(*,*)\right)=(1,3,*,*,4,*,2,*)(*,*). \]
\end{itemize}
\end{exs}

Given any element $\lambda\in\Lambda(m)$, we will freely move between viewing it as an $m$-marked cycle shape or as an element of $\left(\mathfrak{S}_{m}\ltimes_{\varphi}\mathcal{U}_{m}\right)\times\mathcal{C}$, i.e. as an expression of the form $\pi u^{\bm{d}}c^{\bm{l}}$. As we know, a permutation $\sigma$ belongs to the $m$-class $\mathsf{CL}_{\mathbb{N},m}(\lambda)$ if and only if $\sigma$ may be obtained from $\lambda$ by replacing the symbols $*$ with distinct elements from $\mathbb{N}\backslash [m]$. When expressing $\lambda=\pi u^{\bm{d}}c^{\bm{l}}$, then from the above discussion it is clear that we have the following equivalent criteria for when $\sigma$ belongs to $\mathsf{CL}_{\mathbb{N},m}(\lambda)$ given in terms of $\pi$, $\bm{d}$, and $\bm{l}$.

\begin{lem} \label{Lem:mClassCriteria}
We have $\sigma\in\mathsf{CL}_{\mathbb{N},m}(\pi u^{\bm{d}}c^{\bm{l}})$ if and only if the following hold:
\begin{itemize}
\item[(i)] the number of cycles of $\sigma$ of length $i+1$ containing no elements of $[m]$ is $\bm{l}(i)$,
\item[(ii)] $\sigma^{\bm{d}(i)+1}(i)=\pi(i)$ for each $i\in [m]$, and $\sigma^{n}(i)\not\in [m]$ for any $1\leq n \leq \bm{d}(i)$.
\end{itemize}
\begin{flushright} $\square$ \end{flushright}
\end{lem}

Item $(ii)$ above gives a concrete way of describing the relative positions of the elements of $[m]$ within a permutation $\sigma$, and is completely captured by $\pi$ and $\bm{d}$. Recall the quantities $||\lambda||$, $||\lambda||_{m}$, and $||\lambda||^{m}$ for an $m$-marked cycle shape $\lambda\in\Lambda(m)$. Then when expressing $\lambda=\pi u^{\bm{d}}c^{\bm{l}}$, one can deduce that
\[
||\lambda|| = |\mathsf{Sup}(\pi)|+\sum_{i\in [m]}\bm{d}(i)+\sum_{i\in\mathbb{N}}(i+1)\bm{l}(i), \hspace{3mm}
||\lambda||^{m} = \sum_{i\in [m]}\bm{d}(i)+\sum_{i\in\mathbb{N}}(i+1)\bm{l}(i), \hspace{3mm}
||\lambda||_{m} = |\mathsf{Sup}(\pi)|.
\]
Recall the degree function $\mathsf{deg}_{m}:\Lambda(m)\rightarrow\mathbb{Z}_{\geq m}$ given by $\mathsf{deg}_{m}(\lambda)=||\lambda||^{m}+m$. The codomain $\mathbb{Z}_{\geq m}$ is a monoid under the addition $a+_{m}b:=a+b-m$ for all $a,b\in\mathbb{Z}_{\geq m}$ with $m$ the identity. Viewing $\Lambda(m)$ as a monoid, then by consulting \emph{\Cref{mMarkedCycleShapeProduct}} one can deduce that $\mathsf{deg}_{m}$ is a monoid homomorphism, and hence provides a grading for $\Lambda(m)$. For $n\in\mathbb{Z}_{\geq 0}$ define $\Lambda_{n}(m):=\{\lambda\in\Lambda(m) \ | \ \mathsf{deg}_{m}(\lambda)=n\}$ to be the $n$-th graded component of $\Lambda(m)$. We have that $\Lambda_{n}(m)$ is non-empty if and only if $n\geq m$. We have the disjoint unions
\[ \Lambda(m)=\bigsqcup_{n\geq m}\Lambda_{n}(m), \hspace{10mm} \Lambda_{\leq n}(m)=\bigsqcup_{n\geq i\geq m}\Lambda_{i}(m). \]
We now describe generating functions which record the cardinalities of $\Lambda_{n}(m)$ and $\Lambda_{\leq n}(m)$. For a generating function $F(t)$ in a formal variable $t$, we let $[t^{n}]F(t)$ denote the coefficient of the $n$-th degree term $t^{n}$.

\begin{prop} \label{XMCSGenFuncProp}
Let $t$ be a formal variable. As generating functions we have
\begin{equation} \label{ProperXMCSGenFunc}
\sum_{n=0}^{\infty}|\Lambda_{n}(m)|t^{n} = \frac{m!t^{m}}{(1-t)^{m}}\prod_{n=2}^{\infty}\frac{1}{1-t^{n}},
\end{equation}
and
\begin{equation} \label{XMCSGenFunc}
\sum_{n=0}^{\infty}|\Lambda_{\leq n}(m)|t^{n} = \frac{m!t^{m}}{(1-t)^{m}}\prod_{n=1}^{\infty}\frac{1}{1-t^{n}}.
\end{equation}
\end{prop}

\begin{proof}
We begin by showing \emph{\Cref{ProperXMCSGenFunc}}. For any $\pi u^{\bm{d}}c^{\bm{l}}\in\Lambda(m)$ we have that
\[ \mathsf{deg}(\pi u^{\bm{d}}c^{\bm{l}})= m+\sum_{i\in [m]}\bm{d}(i)+\sum_{i\in\mathbb{N}}(i+1)\bm{l}(i). \]
For any $x\in\mathbb{Z}_{\geq 0}$ define the subsets of $\Lambda(m)$ given by
\[ \mathfrak{S}\mathcal{U}^{(x)}:=\left\{ \pi u^{\bm{d}} \ \Bigg| \ \pi\in\mathfrak{S}_{m}, \ \bm{d}\in\mathbb{Z}_{\geq 0}^{[m]}, \ \sum_{i\in [m]}\bm{d}(i)=x\right\} \hspace{2mm}
\text{ and } \hspace{2mm} \mathcal{C}^{(x)}:=\left\{ c^{\bm{l}} \ \Bigg| \ \bm{l}\in\mathbb{Z}_{\geq 0}^{\mathbb{N}}, \ \sum_{i\in\mathbb{N}}(i+1)\bm{l}(i)=x\right\}. \] 
For any $a,b\in\mathbb{Z}_{\geq 0}$, let $\pi u^{\bm{d}}\in\mathfrak{S}\mathcal{U}^{(a)}$ and $c^{\bm{l}}\in\mathcal{C}^{(b)}$, then we have $\mathsf{deg}_{m}(\pi u^{\bm{d}}c^{\bm{l}})=a+b+m$, and all elements in $\Lambda(m)$ of degree $a+b+m$ appear uniquely in such a manner. Hence for any $n\in\mathbb{Z}_{\geq 0}$ we have that
\[ |\Lambda_{n+m}(m)| = \sum\limits_{\substack{a,b\geq 0 \\ a+b=n}}|\mathfrak{S}\mathcal{U}^{(a)}||\mathcal{C}^{(b)}|. \]
Therefore, if $F(t)$ is the generating function such that $[t^{n}]F(t)=|\mathfrak{S}\mathcal{U}^{(n)}|$, and $G(t)$ the generating function such that $[t^{n}]G(t)=|\mathcal{C}^{(n)}|$, then
\begin{equation} \label{CoeffProd}
[t^{n}](t^{m}F(t)G(t))=|\Lambda_{n}(m)|.
\end{equation}
The elements of $\mathcal{C}^{(n)}$ are the $m$-marked cycle shapes where the elements of $[m]$ belong to cycles of length one, and the remaining cycles contain $n$ symbols $*$ in total. The cycles containing only the symbols $*$ must be of length at least two. Thus one can deduce that
\[ G(t)=\sum_{n=0}^{\infty}|\mathcal{C}^{(n)}|t^{n} = \prod_{n=2}^{\infty}\frac{1}{1-t^n}, \]
since the factor $(1-t^{n})^{-1}$ accounts for the number of cycles of length $n$ containing only the symbols $*$ present in such an $m$-marked cycle shape. As for the set $\mathfrak{S}\mathcal{U}^{(n)}$, it is clear that its cardinality is $m!$ times the number of maps $\bm{d}\in\mathbb{Z}_{\geq 0}^{[m]}$ such that the sum of $\bm{d}(i)$ for each $i\in [m]$ is precisely $n$. The number of such maps is precisely the coefficient of $t^{n}$ in the generating function $(1-t)^{-m}$, since each factor $(1-t)^{-1}$ accounts for the choice of the image of a given element of $[m]$. As such
\[ F(t)=\sum_{n=0}^{\infty}|\mathfrak{S}\mathcal{U}^{(n)}|t^{n}=m!\left(\frac{1}{1-t}\right)^{m}. \]
Hence from \emph{\Cref{CoeffProd}} we have that
\[ \sum_{n=0}^{\infty}|\Lambda_{n}(m)|t^{n}= t^{m}m!\left(\frac{1}{1-t}\right)^{m}\prod_{n=2}^{\infty}\frac{1}{1-t^n}, \]
which is precisely \emph{\Cref{ProperXMCSGenFunc}}. For \emph{\Cref{XMCSGenFunc}}, this follows from \emph{\Cref{ProperXMCSGenFunc}} by noting that for any generating function $A(t)$, the new generating function $(1-t)^{-1}A(t)$ records the partial sums of the coefficients of $A(t)$, that is $[t^{n}]\left((1-t)^{-1}A(t)\right)$ equals the sum of $[t^{i}]A(t)$ as $i$ runs from $0$ to $n$.

\end{proof}

Let $\mathcal{P}_{n}$ is the set of all partitions of $n$. Then the above proposition allows us to give a formula for the size of $\Lambda_{\leq n}(m)$ in terms of the sizes of the sets $\mathcal{P}_{i}$ for $i\leq n$.  

\begin{cor} \label{SizeFormulaXMCS}
The cardinality of $\Lambda_{\leq n}(m)$ is given by
\[ |\Lambda_{\leq n}(m)|=\sum\limits_{\substack{a\geq m, b\geq 0 \\ a+b=n}}m!\binom{a-1}{a-m}|\mathcal{P}_{b}|. \]
\end{cor}

\begin{proof}
The following two identities are well-known:
\[ \sum_{n=0}^{\infty}|\mathcal{P}_{n}|t^{n}= \prod_{n=1}^{\infty}\frac{1}{1-t^{n}} \hspace{6mm} \text{ and } \hspace{6mm} \left(\frac{1}{1-t}\right)^{m} = \sum_{n=0}^{\infty}\binom{m+n-1}{n}t^{n}. \]
The latter implies 
\[ \left(\frac{t}{1-t}\right)^{m} = \sum_{n=0}^{\infty}\binom{m+n-1}{n}t^{n+m} = \sum_{n=m}^{\infty}\binom{n-1}{n-m}t^{n}. \]
Then the result follows from \emph{\Cref{XMCSGenFunc}} of \emph{Proposition \ref{XMCSGenFuncProp}}. Note we are using the generalised binomial coefficients here when $m=0$.

\end{proof}

Lastly, with the identification $\Lambda(m)=\left(\mathfrak{S}_{m}\ltimes_{\varphi}\mathcal{U}_{m}\right)\times\mathcal{C}$, we have that
\[ \mathsf{FH}_{m}=\mathsf{Span}_{R}\{K(\pi u^{\bm{d}}c^{\bm{l}}) \ | \ \pi\in\mathfrak{S}_{m}, \bm{d}\in\mathbb{Z}_{\geq 0}^{[m]}, \bm{l}\in\mathbb{Z}_{\geq 0}^{\mathbb{N}}\}. \]
Note that the $R$-subalgebra generated by the elements $K(c^{\bm{l}})$ for all $\bm{l}\in\mathbb{Z}_{\geq 0}^{\mathbb{N}}$ is precisely the subspace $\mathsf{FH}_{m}^{*}$ described in the proof of \emph{\Cref{Lem:FH0SubAlgebra}} consisting of all basis elements indexed by $m$-marked cycles shapes with the elements of $[m]$ appearing in trivial cycles. As such $\mathsf{FH}_{0}\cong\mathsf{Span}_{R}\{K(c^{\bm{l}}) \ | \ \bm{l}\in\mathbb{Z}_{\geq 0}^{\mathbb{N}}\}\subset\mathsf{FH}_{m}$.

%%%%%%%%%%%%%%%%    LEADING TERMS %%%%%%%%%%%%%%%%%%%%%%%%%%%%%%%%%%%%%%%%%%%%%%%%%%%%%%%%%%

\section{Leading Terms via Monoid Product}

We may extend the degree function $\mathsf{deg}_{m}:\Lambda(m)\rightarrow\mathbb{Z}_{\geq m}$ to one acting on $\mathsf{FH}_{m}$ by setting
\[ \mathsf{deg}\left(\sum_{\lambda\in\Lambda(m)}f_{\lambda}(t)K(\lambda)\right)=\mathsf{max}\{\mathsf{deg}_{m}(\lambda) \ | \ f_{\lambda}(t)\neq 0\}. \]
In particular $\mathsf{deg}_{m}(K(\lambda))=\mathsf{deg}_{m}(\lambda)$. We seek to show that the product of two basis elements of $\mathsf{FH}_{m}$ admits a unquie leading term with regards to this degree function. To prove this we will use the following lemma.

\begin{lem} \label{Lem:mSupTrivialIntersection}
For $\lambda,\mu\in\Lambda(m)$, let $g\in\mathsf{CL}_{\mathbb{N},m}(\lambda)$ and $h\in\mathsf{CL}_{\mathbb{N},m}(\mu)$. The equality
\begin{equation} \label{TrivialmSupIntersection}
\mathsf{Sup}^{m}(g)\cap\mathsf{Sup}^{m}(h)=\emptyset
\end{equation}
holds if and only if $gh\in\mathsf{CL}_{\mathbb{N},m}(\lambda\mu)$.
\end{lem}

\begin{proof}
Suppose that $gh\in\mathsf{CL}_{\mathbb{N},m}(\lambda\mu)$, then since $\mathsf{deg}_{m}$ is a monoid homomorphism
\[ |\mathsf{Sup}^{m}(gh)|=||\lambda\mu||^{m}=||\lambda||^{m}+||\mu||^{m}=|\mathsf{Sup}^{m}(g)|+|\mathsf{Sup}^{m}(h)|. \]
Since also $\mathsf{Sup}^{m}(gh)\subset\mathsf{Sup}^{m}(g)\cup\mathsf{Sup}^{m}(h)$, we have $\mathsf{Sup}^{m}(g)\cap\mathsf{Sup}^{m}(h)=\emptyset$. This proves the ``if'' statement. For the ``only if'' let $\lambda=\pi u^{\bm{d}}c^{\bm{l}}$ and $\mu=\sigma u^{\bm{e}}c^{\bm{k}}$ for $\pi,\sigma\in\mathfrak{S}_{m}$, $\bm{d},\bm{e}\in\mathbb{Z}_{\geq 0}^{[m]}$, and $\bm{l},\bm{k}\in\mathbb{Z}_{\geq 0}^{\mathbb{N}}$. By \emph{\Cref{mMarkedCycleShapeProduct}},
\[ \lambda\mu = \pi\sigma u^{\sigma\circ\bm{d}+\bm{e}}c^{\bm{l}+\bm{k}}. \]
Hence the result follows if we show that $gh$ satisfies items $(i)$ and $(ii)$ of \emph{\Cref{Lem:mClassCriteria}} with respect to $\pi\sigma u^{\sigma\circ\bm{d}+\bm{e}}c^{\bm{l}+\bm{k}}$. For item $(i)$, since \emph{\Cref{TrivialmSupIntersection}} is upheld, it is clear that the number of cycles of $gh$ of length $i+1$ which contain no elements of $[m]$ is $\bm{l}(i)+\bm{k}(i)$, since this is the sum of such cycles of $g$ and $h$. For item $(ii)$, pick any $x\in [m]$ and set $y:=\sigma(x)$ and $z:=\pi(y)$. Since $h\in\mathsf{CL}_{\mathbb{N},m}(\sigma u^{\bm{e}}c^{\bm{k}})$, item $(ii)$ of \emph{\Cref{Lem:mClassCriteria}} tells us that $h:x\mapsto i_{1}\mapsto i_{2}\mapsto \cdots \mapsto i_{\bm{e}(x)}\mapsto y$ where $\{i_{1},i_{2},\dots,i_{\bm{e}(x)}\}\cap [m]=\emptyset$. Similarly since $g\in\mathsf{CL}_{\mathbb{N},m}(\pi u^{\bm{d}}c^{\bm{l}})$, $g:y\mapsto j_{1}\mapsto j_{2}\mapsto \cdots \mapsto j_{\bm{d}(y)}\mapsto z$ where $\{j_{1},j_{2},\dots,j_{\bm{d}(y)}\}\cap [m]=\emptyset$. Since \emph{\Cref{TrivialmSupIntersection}} is upheld, we must have that
\[ gh:x\mapsto i_{1}\mapsto \cdots \mapsto i_{\bm{e}(x)}\mapsto j_{1}\mapsto \cdots j_{\bm{d}(y)}\mapsto z. \]
Thus $(gh)^{\bm{d}(y)+\bm{e}(x)+1}(x)=(\pi\sigma)(x)$ and $(gh)^{n}(x)\not\in [m]$ for any $1\leq n\leq \bm{d}(y)+\bm{e}(x)$. One may also note that $\bm{d}(y)=\bm{d}(\sigma^{-1}(x))=(\sigma\circ\bm{d})(x)$, and hence $gh$ also upholds item $(ii)$ of \emph{\Cref{Lem:mClassCriteria}}.

\end{proof}

\begin{prop} \label{Prop:LeadingTerm}
Let $\lambda=\pi u^{\bm{d}}c^{\bm{l}}, \mu=\sigma u^{\bm{e}}c^{\bm{k}}\in\Lambda(m)$. In $\mathsf{FH}_{m}$ we have that
\[ K(\lambda)K(\mu)=c_{\lambda,\mu}K(\lambda\mu)+\sum\limits_{\substack{\nu\in\Lambda(m) \\ \mathsf{deg}_{m}(\nu)<\mathsf{deg}_{m}(\lambda\mu)}}f_{\lambda,\mu}^{\nu}(t)K(\nu) \]
where $c_{\lambda,\mu}\in\mathbb{N}$ is a constant given by
\[ c_{\lambda,\mu}=\prod_{i=1}^{\infty}\binom{(\bm{l}+\bm{k})(i)}{\bm{l}(i)}, \]
\end{prop}

\begin{proof}
Let $g\in\mathsf{CL}_{\mathbb{N},m}(\lambda), h\in\mathsf{CL}_{\mathbb{N},m}(\mu)$, then $\mathsf{Sup}^{m}(gh)\subset\mathsf{Sup}^{m}(g)\cup\mathsf{Sup}^{m}(h)$. Thus if $gh\in\mathsf{CL}_{\mathbb{N},m}(\nu)$ for $\nu\in\Lambda(m)$ then $||\nu||^{m}\leq ||\lambda||^{m}+||\mu||^{m}$. Hence $\mathsf{deg}_{m}(\nu)\leq \mathsf{deg}_{m}(\lambda)+_{m}\mathsf{deg}_{m}(\mu)=\mathsf{deg}_{m}(\lambda\mu)$. Thus we have 
\[ K(\lambda)K(\mu)=\sum\limits_{\substack{\nu\in\Lambda(m) \\ \mathsf{deg}_{m}(\nu)\leq\mathsf{deg}_{m}(\lambda\mu)}}f_{\lambda,\mu}^{\nu}(t)K(\nu). \] 
Suppose $gh\in\mathsf{CL}_{\mathbb{N},m}(\nu)$ with $\mathsf{deg}_{m}(\nu)=\mathsf{deg}_{m}(\lambda\mu)$. This implies that $||\nu||^{m}=||\lambda\mu||^{m}$ and hence we must have $\mathsf{Sup}^{m}(g)\cap\mathsf{Sup}^{m}(h)=\emptyset$. Thus by \emph{\Cref{Lem:mSupTrivialIntersection}} we have that $\nu=\lambda\mu$. Therefore
\[ K(\lambda)K(\mu)=f_{\lambda,\mu}^{\lambda\mu}(t)K(\lambda\mu)+\sum\limits_{\substack{\nu\in\Lambda(m) \\ \mathsf{deg}_{m}(\nu)<\mathsf{deg}_{m}(\lambda\mu)}}f_{\lambda,\mu}^{\nu}(t)K(\nu). \]
It remains to show $f_{\lambda,\mu}^{\lambda\mu}(t)=c_{\lambda,\mu}$. Let $\omega\in\mathsf{CL}_{\mathbb{N},m}(\lambda\mu)$ and $A_{\lambda,\mu}(\omega):=\{(g,h)\in\mathsf{CL}_{\mathbb{N},m}(\lambda)\times\mathsf{CL}_{\mathbb{N},m}(\mu) \ | \ gh=\omega\}$. By \emph{\Cref{mMarkedCycleShapeProduct}} we have $\lambda\mu = \pi\sigma u^{\sigma\circ\bm{d}+\bm{e}}c^{\bm{l}+\bm{k}}$, and so by \emph{\Cref{Lem:mClassCriteria}}, for any $x\in [m]$, we have that $\omega:x\mapsto i_{1}\mapsto \cdots \mapsto i_{(\sigma\circ\bm{d}+\bm{e})(x)}\mapsto (\pi\sigma)(x)$, where $\{i_{1},\cdots,i_{(\sigma\circ\bm{d}+\bm{e})(x)}\}\cap [m]=\emptyset$. Any pair $(g,h)\in A_{\lambda,\mu}(\omega)$ satisfies \emph{\Cref{TrivialmSupIntersection}}, and since their product gives $\omega$ we must have that
\begin{align*}
h&:x\mapsto i_{1}\mapsto \cdots \mapsto i_{\bm{e}(x)}\mapsto \sigma(x), \\
g&:\sigma(x)\mapsto i_{\bm{e}(x)+1}\mapsto \cdots \mapsto i_{(\sigma\circ\bm{d}+\bm{e})(x)}\mapsto (\pi\sigma)(x).
\end{align*}
So if we construct a pair $(g,h)\in\mathsf{CL}_{\mathbb{N},m}(\lambda)\times\mathsf{CL}_{\mathbb{N},m}(\mu)$ such that $gh=\omega$, the cycles containing elements of $[m]$ in $g$ and $h$ are predetermined by $\omega$. Hence we are just concerned with the cycles which contain no elements of $[m]$. In $\omega$ there are $(\bm{l}+\bm{k})(i)$ number of such cycles of length $i+1$, while $g$ and $h$ containg $\bm{l}(i)$ and $\bm{k}(i)$ such cycles respectively. Thus to construct a pair $(g,h)\in A_{\lambda,\mu}(\omega)$, it is simply a matter of how one distributes the cycles containing no elements of $[m]$ of $\omega$ among either $g$ or $h$. The binomial coefficient
\[ \binom{(\bm{l}+\bm{k})(i)}{\bm{l}(i)} \]
counts the number of ways to allocate such cycles of length $i+1$ of $\omega$ to the permutation $g$ (with the remaining cycles allocated to $h$). Therefore
\[ f_{\lambda,\mu}^{\lambda\mu}(t)=|A_{\lambda,\mu}(\omega)|= \prod_{i=1}^{\infty}\binom{(\bm{l}+\bm{k})(i)}{\bm{l}(i)}, \]
where the product is only formally infinite since the functions $(\bm{l}+\bm{k})$, $\bm{l}$, and $\bm{k}$ have finite support.

\end{proof}

Thus the leading term in the product $K(\lambda)K(\mu)$ is $c_{\lambda,\mu}K(\lambda\mu)$. Thus the monoid product of $\Lambda(m)$ is governing the leading terms in $\mathsf{FH}_{m}$. From the formula given above one can deduce that $c_{\lambda,\mu}=1$ if and only if $\lambda$ and $\mu$ share no cycles of the same length consisting of only the symbols $*$, i.e. whenever $\bm{l}(i)\neq 0$ then $\bm{k}(i)=0$. If one was to extend $R$ to a ring which contains the inverses to any natural number, say its field of fractions $\mathbb{Q}(t)$, then the above leading term result could by applied to easily deduce generating sets for the $\mathbb{Q}(t)$-algebra $\mathbb{Q}(t)\otimes_{R}\mathsf{FH}_{m}$. For example, let $s_{i}=(i,i+1)$ for any $i\in[m]$ be the simple transposition in $\mathfrak{S}_{m}$, then by arguing by induction on the degree, \emph{\Cref{Prop:LeadingTerm}} can be employed to show that the set
\begin{equation} \label{GeneratingSetOverQ(t)}
\{K(s_{i}),K(u_{i}),K(c_{j}) \ | \ i\in[m], j\in\mathbb{N}\}
\end{equation}
generates $\mathbb{Q}(t)\otimes_{R}\mathsf{FH}_{m}$ as a $\mathbb{Q}(t)$-algebra since $\{s_{i},u_{i},c_{j} \ | \ i\in[m], j\in\mathbb{N}\}$ generates $\Lambda(m)=\left(\mathfrak{S}_{m}\ltimes_{\varphi}\mathcal{U}_{m}\right)\times\mathcal{C}$.

%%%%%%%%%%%%%%%%    SYMMETRIC FUNCTIONS %%%%%%%%%%%%%%%%%%%%%%%%%%%%%%%%%%%%%%%%%%%%%%%%%%%%%%%

\section{Symmetric Functions}

Consider the polynomial $\mathbb{Z}$-algebra $\mathbb{Z}[x_{1},\dots,x_{n}]$ in $n$ commuting variables. The group $\mathfrak{S}_{n}$ acts on this algebra by permuting the variables. A polynomial is \emph{symmetric} if it is $\mathfrak{S}_{n}$-invariant. Denote the $\mathbb{Z}$-subalgebra of symmetric polynomials by $\mathsf{Sym}_{n}$. We are interested in two types of symmetric polynomials, and to describe one type it will be helpful to set up some notation. A tuple $\alpha=(a_{1},a_{2},\dots,a_{l})$ of positive integers is a partition whenever $a_{1}\geq a_{2}\geq \cdots \geq a_{l}$. We say $\alpha$ has length $l$ and size $a_{1}+\cdots+a_{l}$. Let $\mathcal{P}(k,l)$ be the set of partitions of size $k$ and length $l$. When $k=l=0$ let $\mathcal{P}(0,0)=\{\emptyset\}$.

\begin{defn}
For any $k\geq 0$, the \emph{elementary symmetric polynomials} are given by
\[ e_{k}(x_{1},\dots,x_{n}):=\sum_{i_{1}<i_{2}<\cdots<i_{k}}x_{i_{1}}x_{i_{2}}\cdots x_{i_{k}} \]
Hence $e_{k}$ is the sum of all monomials with $k$ variables. In particular $e_{0}=1$ and $e_{k}=0$ for $k>n$.
\end{defn}

\begin{defn}
For $\alpha=(a_{1},a_{2},\dots,a_{l})\in\mathcal{P}(k,l)$, the \emph{monomial symmetric polynomials} are given by
\[ m_{\alpha}(x_{1},\dots,x_{n}):=\sum_{(i_{1},i_{2},\dots,i_{l})\in [n]^{!l}}x_{i_{1}}^{a_{1}}x_{i_{2}}^{a_{2}}\cdots x_{i_{l}}^{a_{l}} \]
Recall that $[n]^{!l}$ is the subset of the $l$-fold cartesian product of $[n]$ consisting of tuples with pairwise distinct entries. Hence $m_{\alpha}$ is the sum of all monomials whose exponents match the partition $\alpha$ up to rearrangement. In particular $m_{\emptyset}=1$ and $m_{\alpha}=0$ whenever $l>n$.
\end{defn}  

It is well known that $\mathsf{Sym}_{n}$ is generated (transcendentally) by $e_{1},e_{2},\dots,e_{n}$. That is we have a $\mathbb{Z}$-algebra isomorphism $\mathsf{Sym}_{n}\cong \mathbb{Z}[e_{1},\dots,e_{n}]$ with the latter being a free polynomial algebra in $n$ commuting generators. One may show that $\{m_{\alpha} \ | \ \alpha\in\mathcal{P}(k,l), k\geq 0, n\geq l \geq 0\}$ forms a $\mathbb{Z}$-basis for $\mathsf{Sym}_{n}$. Assigning each variable $x_{i}$ a degree of 1, then let $\mathsf{Sym}_{n}^{k}$ denote the degree $k$ component of $\mathsf{Sym}_{n}$. In particular $e_{k},m_{\alpha}\in\mathsf{Sym}_{n}^{k}$ for any $\alpha\in\mathcal{P}(k,l)$. For any $N>n$ we have a surjective $\mathbb{Z}$-module homomorphisms $\rho_{N,n}:\mathsf{Sym}_{N}^{k}\rightarrow\mathsf{Sym}_{n}^{k}$ given by evaluating the variables $x_{n+1},\dots,x_{N}$ at zero. One can show that the collection of such morphisms defines an inverse system for the $\mathbb{Z}$-modules $\mathsf{Sym}_{n}^{k}$ (with fixed $k$), and so we have the inverse limit
\[ \mathsf{Sym}^{k}:=\varprojlim\mathsf{Sym}_{n}^{k}. \]
Any element of $\mathsf{Sym}^{k}$ is of the form $(f_{1},f_{2},\dots)$ with $f_{n}\in\mathsf{Sym}_{n}^{k}$ and where $\rho_{N,n}(f_{N})=f_{n}$ for all $N>n$. One can note that $\rho_{N,n}(m_{\alpha}(x_{1},\dots,x_{N}))=m_{\alpha}(x_{1},\dots,x_{n})$ for any $N>n$ and $\alpha\in\mathcal{P}(k,l)$. Thus we write $m_{\alpha}=(m_{\alpha}(x_{1}),m_{\alpha}(x_{1},x_{2}),\dots)$ and call such elements of $\mathsf{Sym}^{k}$ the \emph{monomial symmetric functions}. The same can be said for $e_{k}$ since $e_{k}=m_{(1^{k})}$ where $(1^{k})$ is the partition of size $k$ consisting of $k$ parts equal to 1. The $\mathbb{Z}$-algebra of \emph{symmetric functions} is given by
\[ \mathsf{Sym}:=\bigoplus_{k\in\mathbb{Z}_{\geq 0}}\mathsf{Sym}^{k}. \] 
We have the isomorphism of $\mathbb{Z}$-algebras $\mathsf{Sym}\cong\mathbb{Z}[e_{1},e_{2},\dots]$, and a $\mathbb{Z}$-basis $\{m_{\alpha} \ | \ \alpha\in\mathcal{P}(k,l), k\geq l\geq 0\}$.

%%%%%%%%%%%%%%%%   JM-ELEMENTS  %%%%%%%%%%%%%%%%%%%%%%%%%%%%%%%%%%%%%%%%%%%%%%%%%%%%%%%%%%%%

\section{Jucys-Murphy Elements and Generators of $\mathsf{FH}_{0}$}

In this section we recall the Jucys-Murphy elements of the symmetric group algebras, and their connections to the Farahat-Higman algebras, see also \cite[Section 3]{Ryba21}.

\begin{defn}
For each $1\leq i\leq n$, the $i$-th \emph{Jucys-Murphy element} $L_{i}$ of $\mathbb{Z}\mathfrak{S}_{n}$ is defined by
\[ L_{i}:=\sum_{1\geq j<i}(j,i). \]
\end{defn}

Note that $L_{1}=0$. The following relations regarding the Jucys-Murphy elements are well-known.

\begin{lem} \label{Lem:JMRelations}
The following hold within $\mathbb{Z}\mathfrak{S}_{n}$:
\begin{itemize}
\item[(1)] $L_{i}L_{j}=L_{j}L_{i}$ for all $i,j\in[n]$
\item[(2)] $s_{i}L_{j}=L_{j}s_{i}$ for all $i\in[n-1]$ and $j\neq i,i+1$
\item[(3)] $L_{i+1}=s_{i}L_{i}s_{i}+s_{i}$ for all $i\in[n-1]$
\end{itemize}
where $s_{i}=(i,i+1)$ is the simple transposition exchanging $i$ and $i+1$ for all $i\in[n-1]$.
\end{lem}

%Let $\mathsf{GZ}_{n}=\langle L_{1},\dots,L_{n}\rangle$. In \cite[Proposition 1.1 and Corollary 2.6]{OV96} the $\mathbb{C}$-subalgebra $\mathbb{C}\otimes_{\mathbb{Z}}\mathsf{GZ}%_{n}$ was shown to be a maximal commutative subalgebra of $\mathbb{C}\mathfrak{S}_{n}=\mathbb{C}\otimes_{\mathbb{Z}}\mathbb{Z}\mathfrak{S}_{n}$, and hence the analgous result %is true for the integral setting.
%
%\begin{lem}
%The $\mathbb{Z}$-subalgebra $\mathsf{GZ}_{n}$ is a maximal commutative subalgebra of $\mathbb{Z}\mathfrak{S}_{n}$.
%\end{lem}
%
%\begin{proof}
%Assume $z\in\mathbb{Z}\mathfrak{S}_{n}$ commutes with $\mathsf{GZ}_{n}$, and suppose for contradiction that $z\not\in\mathsf{GZ}_{n}$. Well $1\otimes z\in\mathbb{C}\mathfrak{S}%_{n}$ must commutative with $\mathbb{C}\otimes_{\mathbb{Z}}\mathsf{GZ}_{n}$, which implies that is must belong to $\mathbb{C}\otimes_{\mathbb{Z}}\mathsf{GZ}_{n}$ since it is a %maximal commutative subalgebra, which gives the desired contradiction.
%
%\end{proof}
%
%Therefore any element of the center $Z(\mathbb{Z}\mathfrak{S}_{n})$ must be a certain polynomial in the Jucys-Murphy elements. In fact, 

It was shown in \cite[Theorem 1.9]{Murphy83} that the center is precisley the collection of all symmetric polynomials in the Jucys-Murphy elements. Then from the last section, this means that
\begin{equation} \label{SymPolysJMCentral}
Z(\mathbb{Z}\mathfrak{S}_{n})=\langle e_{1}(L_{1},\dots,L_{n}),\dots,e_{n}(L_{1},\dots,L_{n})\rangle.
\end{equation}
We know the center $Z(\mathbb{Z}\mathfrak{S}_{n})$ has a $\mathbb{Z}$-basis given by the set $\{K_{n}(\lambda) \ | \ \lambda\in\Lambda_{\leq n}(0)\}$. It is natural to ask how the elementary symmetric polynomials in the Jucys-Murphy elements decompose as a linear combination of class sums. This was answered in \cite[Section 3]{Jucys74} as we recall now. Let $\lambda\in\Lambda(0)$ contain $l$ many cycles. We say the \emph{reduced degree} of $\lambda$ is $\mathsf{rd}(\lambda):=\mathsf{deg}_{0}(\lambda)-l$, i.e. the number of symbols $*$ present in $\lambda$ minus the number of cycles. Then it was shown that
\[ e_{k}(L_{1},\dots,L_{n})=\sum\limits_{\substack{\lambda\in\Lambda_{\leq n}(0) \\ \mathsf{rd}(\lambda)=k}}K_{n}(\lambda) \]
for any $k\in[n]$. Recall that $\mathsf{pr}_{n,0}(\mathsf{FH}_{0})=Z(\mathbb{Z}\mathfrak{S}_{n})$, and from the above formula there are natural elements of $\mathsf{FH}_{0}$ to consider which project down to the elementary symmetric polynomials in the Jucys-Murphy elements.

\begin{defn}
For any $k\in\mathbb{Z}_{\geq 0}$ define
\[ E_{k}:=\sum\limits_{\substack{\lambda\in\Lambda(0) \\ \mathsf{rd}(\lambda)=k}}K(\lambda), \]
\end{defn}

Note only finitely many elements of $\Lambda(0)$ have reduced degree $k$, so the above definition is well-defined. These elements project down to the elementary symmetric polynomials in the corresponding Jucys-Murphy elements, and are precisely the elements discussed in the introduction. In particular, they were shown in \cite{FH59} to generate all of $\mathsf{FH}_{0}$, which we record here.

\begin{thm} \label{Thm:FHGenertors} [Theorem 2.5 of \cite{FH59}]
The set of elements $\{E_{1},E_{2},\dots\}$ generates $\mathsf{FH}_{0}$ as an $R$-algebra.
\end{thm}

As was also mentioned in the introduction, it was proven in \cite[Theorem 3.8]{Ryba21} that we have an isomorphism $\mathsf{FH}_{0}\cong R\otimes_{\mathbb{Z}}\mathsf{Sym}$ of $R$-algebras which associates $E_{k}$ to $e_{k}$. In the next section we will prove an analogous result for $\mathsf{FH}_{m}$ in \emph{\Cref{Thm:FHmIsomorphicm}}, and to do so it will be helpful to discuss elements of $\mathsf{FH}_{0}$ which are to the monomial symmetric polynomials $m_{\alpha}$ what the elements $E_{k}$ are to the $e_{k}$.

\begin{lem} \label{Lem:ElementsMAlpha}
For any $k\geq l\geq 0$ and $\alpha\in\mathcal{P}(k,l)$, there exists an element $M_{\alpha}\in\mathsf{FH}_{0}$ such that
\[ \mathsf{pr}_{n,0}(M_{\alpha})=m_{\alpha}(L_{1},\dots,L_{n}). \]
\end{lem}

\begin{proof}
As $\mathsf{Sym}=\mathbb{Z}[e_{1},e_{2},\dots]$ then $m_{\alpha}$ is a finite $\mathbb{Z}$-linear combination of monomials in $e_{1},e_{2},\dots$. Letting $M_{\alpha}$ be obtained from $m_{\alpha}$ by replacing $e_{k}$ with $E_{k}$ gives the element we are looking for.

\end{proof}

We wish to say a little more about the elements $M_{\alpha}$. For any $\alpha=(a_{1},\dots,a_{l})\in\mathcal{P}(k,l)$ let $\overline{\alpha}\in\Lambda(0)$ denote the cycle shape with $l$ cycles of lengths $a_{1}+1,a_{2}+1,\dots,a_{l}+1$. When $\alpha=\emptyset$ then $\overline{\emptyset}=\emptyset$. One can see that $\mathsf{rd}(\overline{\alpha})=k$, the size of $\alpha$. Then in the proof of Theorem $1.9$ of \cite{Murphy83}, see also \cite[Proposition 3.11]{Ryba21}, the following result was proven. Let $\mathcal{P}:=\cup_{k\geq l\geq 0}\mathcal{P}(k,l)$. Note that the map $\overline{(-)}:\mathcal{P}\rightarrow\Lambda(0)$ sending $\alpha\mapsto\overline{\alpha}$ is a bijection. 

\begin{prop}
Let $\alpha\in\mathcal{P}$ be such that $\mathsf{deg}_{0}(\overline{\alpha})\leq n$, then 
\[ m_{\alpha}(L_{1},\dots,L_{n})=K_{n}(\overline{\alpha})+\sum c_{\mu}(n)K_{n}(\mu), \]
where the sum runs over all $\mu\in\Lambda(0)$ such that $\mathsf{rd}(\mu)<\mathsf{rd}(\overline{\alpha})$ or that $\mathsf{rd}(\mu)=\mathsf{rd}(\overline{\alpha})$ and $\mu$ contains less cycles than $\overline{\alpha}$ (noting that only finitely many such $\mu$ exist), and where $c_{\mu}(n)\in\mathbb{Z}_{\geq 0}$.
\end{prop}

Applying \emph{\Cref{Lem:FHEqualityByProjection}} allows us to deduce the following:

\begin{lem}
Let $\alpha\in\mathcal{P}$, then 
\begin{equation} \label{MAlphaLeadingTerm}
M_{\alpha}=K(\overline{\alpha})+\sum c_{\mu}(t)K(\mu),
\end{equation}
where the sum runs over all $\mu\in\Lambda(0)$ such that $\mathsf{rd}(\mu)<\mathsf{rd}(\overline{\alpha})$ or that $\mathsf{rd}(\mu)=\mathsf{rd}(\overline{\alpha})$ and $\mu$ contains less cycles than $\overline{\alpha}$ (noting that only finitely many such $\mu$ exist), and where $c_{\mu}(t)\in R$ are such that $c_{\mu}(n)$ are the coefficients of $K_{n}(\mu)$ in the above proposition.
\end{lem}

We may now show that the elements $M_{\alpha}$ provide an $R$-basis for $\mathsf{FH}_{0}$.

\begin{prop}
The set $\{M_{\alpha} \ | \ \alpha\in\mathcal{P}\}$ gives an $R$-basis for $\mathsf{FH}_{0}$. 
\end{prop}

\begin{proof}
By \emph{\Cref{MAlphaLeadingTerm}} the set $\{M_{\alpha} | \alpha\in\mathcal{P}\}$ is $R$-linearly independent, hence we only need to show that it spans $\mathsf{FH}_{0}$. Equip the $R$-basis $\{K(\lambda) | \lambda\in\Lambda(0)\}$ of $\mathsf{FH}_{0}$ with the partial order $<$ defined by letting $K(\mu)< K(\lambda)$ if $\mathsf{rd}(\mu)<\mathsf{rd}(\lambda)$ or if $\mathsf{rd}(\mu)=\mathsf{rd}(\lambda)$ and $\mu$ contains less cycles than $\lambda$. For any $n\in\mathbb{Z}_{\geq 0}$ let $\mathsf{FH}_{0}^{\leq n}$ denote the $\mathbb{Z}$-module spanned by $K(\lambda)$ such that $\mathsf{rd}(\lambda)\leq n$. We prove $\mathsf{FH}_{0}^{\leq n}\subset\mathsf{Span}_{R}\{M_{\alpha} \ | \ \alpha\in\mathcal{P}\}$ by induction on $n$. The base case $\mathsf{FH}_{0}^{\leq 0}=\mathsf{Span}_{R}\{K(\emptyset)\}$ holds since $M_{\emptyset}=K(\overline{\emptyset})=K(\emptyset)$. Consider $\mathsf{FH}_{0}^{\leq n}$ for some $n\geq 1$, and let
\[ K=\sum\limits_{\substack{\lambda\in\Lambda(0) \\ \mathsf{rd}(\lambda)\leq n}}f_{\lambda}(t)K(\lambda)\in\mathsf{FH}_{0}^{\leq n}. \]
Let $K(\lambda')$  be a maximal element in $\{K(\lambda) \ | \ f_{\lambda}(t)\neq 0\}$ with respect to $<$. Let $\alpha$ be the unique partition of $\mathcal{P}$ such that $\overline{\alpha}=\lambda'$, then $K-f_{\lambda'}(t)M_{\alpha}$ is an element of $\mathsf{FH}_{0}^{\leq n}$ whose terms are strictly less that $K(\lambda')$ in the partial order or incomparable. Continuing this procedure of removing maximal basis elements will lead to an element of $\mathsf{FH}_{0}^{\leq n-1}$, and by induction the result will belong to the $R$-span of $\{M_{\alpha} \ | \ \alpha\in\mathcal{P}\}$. Since we got their by substracting some $R$-linear combination of elements $M_{\alpha}$, the starting element $K$ must also belong to the $R$-span of $\{M_{\alpha} \ | \ \alpha\in\mathcal{P}\}$ which completes the proof by induction. The proposition follows since $\mathsf{FH}_{0}$ is the union of $\mathsf{FH}_{0}^{\leq n}$ for all $n\geq 0$.
 
\end{proof}

%%%%%%%%%%%%%%%%    THE ISOMORPHISM  %%%%%%%%%%%%%%%%%%%%%%%%%%%%%%%%%%%%%%%%%%%%%%%%%%%%%%%%%

\section{The Isomorphism $\mathsf{FH}_{m}\cong R\otimes_{\mathbb{Z}}(\mathcal{H}_{m}\otimes\mathsf{Sym})$}

In this section we recall the definition of the degnerate affine Hecke algebra $\mathcal{H}_{m}$, and some of its structural properties. We define some elements of $\mathsf{FH}_{m}$ which mimic the variables generators of  $\mathcal{H}_{m}$, and end by proving the isomorphism in the section title.

\begin{defn} \label{Defn:DAHA}
The \emph{degenerate affine Hecke Algebra} $\mathcal{H}_{m}$ is the $\mathbb{Z}$-algebra presented with generating set $\{s_{i}, y_{j} \ | \ 1\leq i\leq m-1, \ j\in[m]\}$ and defining relations
\begin{multicols}{2}
\begin{itemize}
\item[(1i)] $s_{i}^{2} = 1$, for $i\in [m-1]$.
\item[(1ii)] $s_{i}s_{j} = s_{j}s_{i}$, for $j\neq i-1, i+1$.
\item[(1iii)] $s_{i}s_{i+1}s_{i} = s_{i+1}s_{i}s_{i+1}$, for $i\in [m-2]$.
\item[(2i)] $y_{i}y_{j}=y_{j}y_{i}$ for all $i,j\in[m]$.
\item[(2ii)] $s_{i}y_{j}=y_{j}s_{i}$ for all $j\neq i, i+1$.
\item[(2iii)] $y_{i+1}=s_{i}y_{i}s_{i}+s_{i}$ for all $i\in[m-1]$.
\end{itemize}
\end{multicols}
\end{defn}

The elements $s_{i}$ are the simple transpositions $(i,i+1)$ of $\mathfrak{S}_{m}$. The algebra $\mathcal{H}_{m}$ has a basis of the form
\begin{equation} \label{DAHABasis}
\{ \pi y_{1}^{\bm{d}(1)}\cdots y_{m}^{\bm{d}(m)} \ | \ \pi\in\mathfrak{S}_{m}, \bm{d}\in\mathbb{Z}_{\geq 0}^{[m]} \}.
\end{equation}
For a proof see \cite[Theorem 3.2.2]{K05}. Hence one can deduce that the subalgebra generated by the elements $s_{i}$ is a copy of $\mathbb{Z}\mathfrak{S}_{m}$, and similarly the subalgebra generated by the variable generators $y_{1},\dots,y_{m}$ is a copy of the polynomial $\mathbb{Z}$-algebra $\mathbb{Z}[y_{1},\dots,y_{m}]$. Moreover, as an $\mathbb{Z}$-module we have $\mathcal{H}_{m}\cong\mathbb{Z}\mathfrak{S}_{m}\otimes\mathbb{Z}[y_{1},\dots,y_{m}]$. By \emph{\Cref{Lem:JMRelations}}, we have a surjective $\mathbb{Z}$-algebra homomorphism $\mathcal{H}_{m}\rightarrow\mathbb{Z}\mathfrak{S}_{m}$ given by $s_{i}\mapsto s_{i}$ and $y_{j}\mapsto L_{j}$.

We now show that the algebra $\mathsf{FH}_{m}$ contains like-minded elements.

\begin{defn}
For any $i\in[m]$ define elements $Y_{i}\in\mathsf{FH}_{m}$ by
\[ Y_{i}:=K(u_{i})+\sum\limits_{\substack{j\in[m] \\ j<i}}K((j,i))=K(u_{i})+L_{i}, \]
where we have identified $R\mathfrak{S}_{m}$ with $\{K(\pi) \ | \ \pi\in\mathfrak{S}_{m}\subset\Lambda(m)\}$ by \emph{\Cref{Lem:SmSubAlgebra}}.
\end{defn}

\begin{ex}
In $\mathsf{FH}_{3}$ we have
\[ Y_{1}=K((*,1)), \hspace{3mm} Y_{2}=K((*,2))+K((1,2)), \hspace{3mm} Y_{3}=K((*,3))+K((1,3))+K((2,3)), \]
where we have dropped the trivial cycles $(1)$, $(2)$, and $(3)$ from each of the $3$-marked cycle shapes appeaing.
\end{ex}

Since $\mathsf{deg}_{m}(u_{i})=m+1$ we have that $\mathsf{pr}_{m,m}(K(u_{i}))=0$, hence $\mathsf{pr}_{m,m}(Y_{i})=L_{i}$. However, one can deduce that when $n>0$ the elements $\mathsf{pr}_{n+m,m}(Y_{i})$ do not get send to the Jucys-Murphy elements $L_{i}$ of $\mathbb{Z}\mathfrak{S}_{n+m}$ but rather $\sigma_{n}L_{n+i}\sigma_{n}^{-1}$ for any $\sigma_{n}\in\mathfrak{S}_{n+m}$ for which $\sigma_{n}([m])=\{n+1,\dots,n+m\}$. Nonetheless they retain counterparts to the relations of \emph{\Cref{Lem:JMRelations}}:

\begin{lem} \label{Lem:JMRelationsFHm}
The following relations hold within $\mathsf{FH}_{m}$:
\begin{itemize}
\item[(1)] $Y_{i}Y_{j}=Y_{j}Y_{i}$ for all $i,j\in[n]$
\item[(2)] $K(s_{i})Y_{j}=Y_{j}K(s_{i})$ for all $i\in[n-1]$ and $j\neq i,i+1$
\item[(3)] $Y_{i+1}=K(s_{i})Y_{i}K(s_{i})+K(s_{i})$ for all $i\in[n-1]$
\end{itemize}
where $s_{i}=(i,i+1)$ is the simple transposition exchanging $i$ and $i+1$ for all $i\in[n-1]$.
\end{lem}

\begin{proof}
$(1)$: For any $n\geq 0$ and $i,j\in[m]$ we have
\[ \mathsf{pr}_{n+m,m}(Y_{i}Y_{j})=\sigma_{n}L_{i+m}L_{j+m}\sigma_{n}^{-1}= \sigma_{n}L_{j+m}L_{i+m}\sigma_{n}^{-1}=\mathsf{pr}_{n+m,m}(Y_{j}Y_{i}), \]
since the Jucys-Murphy elements commute. So applying \emph{\Cref{Lem:FHEqualityByProjection}} shows that $Y_{i}Y_{j}=Y_{j}Y_{i}$. $(2)$: It follows as a straight forward application of \emph{\Cref{Lem:FHEqualityByProjection}} that $K(s_{i})$ commutes with $K(u_{j})$ whenever $i\in[n-1]$ and $j\neq i,i+1$, and we already know that $K(s_{i})$ commutes with $L_{j}$, so $(2)$ follows. $(3)$: Again by  \emph{\Cref{Lem:FHEqualityByProjection}} one can show that $K(s_{i})K(u_{i})K(s_{i})=K(u_{i+1})$, and hence
\[ K(s_{i})Y_{i}K(s_{i})=K(u_{i+1})+s_{i}L_{i}s_{i}=K(u_{i})+L_{i+1}-K(s_{i})=Y_{j}-K(s_{i}), \]
where we used $(3)$ of \emph{\Cref{Lem:JMRelations}}. Rearranging gives $(3)$.

\end{proof}

\begin{thm} \label{Thm:FHmIsomorphicm}
We have an isomorphism of $R$-algebras $\phi:R\otimes_{\mathbb{Z}}(\mathcal{H}_{m}\otimes\mathsf{Sym})\rightarrow\mathsf{FH}_{m}$ defined by $s_{i}\mapsto K(s_{i})$, $y_{j}\mapsto Y_{j}$, and $e_{k}\mapsto E_{k}$ for each $i\in[m-1]$, $j\in[m]$, and $k\geq 0$. 
\end{thm}

\begin{proof}
To prove $\phi$ is a homomorphism it suffices to show that the relations of \emph{\Cref{Defn:DAHA}} are respected, and that the elements $\phi(e_{k})$ commute with $\mathsf{Im}(\phi)$. Relations $(1)$ of \emph{\Cref{Defn:DAHA}} are respected by \emph{\Cref{Lem:SmSubAlgebra}}, and relations $(2)$ are upheld by \emph{\Cref{Lem:JMRelationsFHm}}. Also $\phi(e_{k})=E_{k}\in\mathsf{FH}_{0}\subset\mathsf{FH}_{m}$, hence the images $\phi(e_{1}),\phi(e_{2}),\dots$ commute with one another since $\mathsf{FH}_{0}$ is commutative. Moreover, since they belong to $\mathsf{FH}_{0}$, their projections down to $Z_{n,m}$ by $\mathsf{pr}_{n,m}$ consist of permutations which act trivially on $[m]$. Therefore $\phi(e_{k})$ commutes with $\phi(s_{i})=K(s_{i})$ for any $k\geq 0$ and $i\in[m-1]$. Lastly we have that
\[ [\phi(e_{k}),\phi(y_{j})]=[\phi(e_{k}),Y_{j}]=[\phi(e_{k}),K(u_{j})]. \]
We have that $\mathsf{pr}_{n,m}(K(u_{j}))$ is the sum of transpositions $(a,j)$ for all $a\in[n]\backslash [m]$, which commutes with any permutation which fixes the elements $[m]$. In particular $\mathsf{pr}_{n,m}(K(u_{j}))$ must commute with $\mathsf{pr}_{n,m}(\phi(e_{k}))$, and so by \emph{\Cref{Lem:FHEqualityByProjection}} we have that $K(u_{j})$ commutes with $\phi(e_{k})$. Hence $[\phi(e_{k}),\phi(y_{j})]=0$, and so $\phi(e_{k})$ commutes with $\mathsf{Im}(\phi)$. Thus $\phi$ is an $R$-algebra homomorphism. For surjectivity we show $K(\lambda)\in\mathsf{Im}(\phi)$ for any $\lambda\in\Lambda(m)$. Write $\lambda=\pi u^{\bm{d}}c^{\bm{l}}$ and consider the element $K(c^{\bm{l}})$ which belongs to the $R$-subalgebra $\mathsf{FH}_{0}$ of $\mathsf{FH}_{m}$. By \emph{\Cref{Thm:FHGenertors}} their exists $C\in\mathsf{Sym}=\mathbb{Z}[e_{1},e_{2}\dots]$ such that $\phi(C)=K(c^{\bm{l}})$. Then by employing the leading term result of \emph{\Cref{Prop:LeadingTerm}}, we have that
\[ \phi\left(\pi y_{1}^{\bm{d}(1)}\cdots y_{m}^{\bm{d}(m)} C\right) = K(\pi)Y_{1}^{\bm{d}(1)}\cdots Y_{m}^{\bm{d}(m)}K(c^{\bm{l}})= K(\pi u^{\bm{d}}c^{\bm{l}})+T \]
where $T$ stands for an $R$-linear combination of terms $K(\mu)$ such that $\mathsf{deg}_{m}(\mu)<\mathsf{deg}_{m}(\lambda)$. Hence arguing by induction on the degree of $\lambda$ shows surjectivity, noting that the base case is immediate since the basis elements of degree zero are precisely $K(\pi)=\phi(\pi)$ for some $\pi\in\mathfrak{S}_{m}\subset\Lambda(m)$. For injectivity, by \emph{\Cref{DAHABasis}} and recalling that the monomial symmetric functions form a $\mathbb{Z}$-basis of $\mathsf{Sym}$, we have that the set
\[ \mathsf{B}:=\left\{\pi y_{1}^{\bm{d}(1)}\cdots y_{m}^{\bm{d}(m)}\otimes m_{\alpha} \ | \ \pi\in\mathfrak{S}_{m}, \bm{d}\in\mathbb{Z}_{\geq 0}^{[m]}, \alpha\in\mathcal{P}\right\} \]
forms an $R$-basis of $R\otimes_{\mathbb{Z}}(\mathcal{H}_{m}\otimes\mathsf{Sym})$. We seek to show that $\phi(\mathsf{B})$ is $R$-linearly independent. Equip the basis set $\{K(\pi u^{\bm{d}}c^{\bm{l}}) \ | \ \pi\in\mathfrak{S}_{m}, \bm{d}\in\mathbb{Z}_{\geq 0}^{[m]}, \bm{l}\in\mathbb{Z}_{\geq 0}^{\mathbb{N}}\}$ of $\mathsf{FH}_{m}$ with the partial order $<$ define by $K(\sigma u^{\bm{e}}c^{\bm{k}})<K(\pi u^{\bm{d}}c^{\bm{l}})$ if $(i)$ the degree of $\sigma u^{\bm{e}}c^{\bm{k}}$ is strictly less than that of $\pi u^{\bm{d}}c^{\bm{l}}$, $(ii)$ their degrees agree but $\mathsf{rd}(c^{\bm{k}})<\mathsf{rd}(c^{\bm{l}})$, $(iii)$ their degrees agree and $\mathsf{rd}(c^{\bm{k}})=\mathsf{rd}(c^{\bm{l}})$ but $c^{\bm{k}}$ contains less cycles than $c^{\bm{l}}$. Note \emph{\Cref{Prop:LeadingTerm}} tells us that the product of $K(\lambda)$ and $K(\mu)$ in $\mathsf{FH}_{m}$ results in $c_{\lambda,\mu}K(\lambda\mu)$ plus addition terms all lower in the order $<$ where $c_{\lambda,\mu}\in\mathbb{N}$. Recall that $c_{\lambda,\mu}=1$ whenever $\lambda$ and $\mu$ share no cycles of the same size which contain only the symbol $*$. From \emph{\Cref{MAlphaLeadingTerm}} and \emph{\Cref{Prop:LeadingTerm}} we have that $\phi$ acts on an element of $\mathsf{B}$ by
\[ \phi\left(\pi y_{1}^{\bm{d}(1)}\cdots y_{m}^{\bm{d}(m)}\otimes m_{\alpha}\right)=K(\pi)Y_{1}^{\bm{d}(1)}\cdots Y_{m}^{\bm{d}(m)}M_{\alpha}=K(\pi u^{\bm{d}}\overline{\alpha})+T \]
where $T$ is an $R$-linear combination of basis elements $K(\lambda)$ of $\mathsf{FH}_{m}$ which are strictly less with respect to $<$. As such, in the image of any finite $R$-linear combinaition of elements of $\mathsf{B}$ under $\phi$, we may pick out a non-zero term which is incomparable or strictly greater than any other term with respect to $<$, showing that $\phi(\mathsf{B})$ is $R$-linearly independent in $\mathsf{FH}_{m}$, which proves that $\phi$ is also injective.

\end{proof}

\begin{rmk}
The $\mathbb{C}$-algebra $\mathbb{C}\otimes_{\mathbb{Z}}(\mathcal{H}_{m}\otimes\mathsf{Sym})$ and close variations have made appearences within the literature. In \cite{MO01}, they aimed to give a centraliser construction for the degenerate affine Hecke algebra in a manner comparable to how the Yangians arise from a projective limit of universal enveloping algebras of $\mathfrak{gl}_{n}$. No such projective system exists for the group algebras of the symmetric groups, so instead they work with the larger semigroup of partial permutations. Algebras $A_{m}:=A_{0}\otimes\tilde{\mathcal{H}}_{m}$ were constructed, where $\tilde{\mathcal{H}}_{m}$ is a degenerate affine counterpart to the semigroup algebra of partial permutations, and $A_{0}$ was shown to be isomorphic to the algebra of shifted symmetric functions. In our setting the lack of a projective system was sidestepped by employing the techniques of Farahat and Higman on the centraliser algebras $Z_{n,m}$, which allowed us to stay working with the symmetric group itself. Also $\mathbb{C}\otimes_{\mathbb{Z}}(\mathcal{H}_{m}\otimes\mathsf{Sym})$ appears in the Heisenberg category $\mathsf{Heis}$ of M. Khovanov defined in \cite{Kho14}. Such a category is a monoidal category generated by two objects $\uparrow$ and $\downarrow$, and where the morphism spaces are defined diagrammatically. In \cite[Proposition 4]{Kho14} it was shown that the endomorphism algebra $\mathsf{End}_{\mathsf{Heis}}(\uparrow^{\otimes m})$ is isomorphic to $\mathbb{C}\otimes_{\mathbb{Z}}(\mathcal{H}_{m}\otimes\mathsf{Sym})$.
\end{rmk}

We collect some consequences of the above theorem.

\begin{cor}
The algebra $\mathsf{FH}_{m}$ is generated by $K(s_{i})$, $Y_{j}$, and $E_{k}$ for $i\in[m-1]$, $j\in[m]$, and $k\geq 0$.
\end{cor}

\begin{cor}
The algebra $\mathsf{FH}_{m}$ has an $R$-basis given by the set
\[ \left\{ K(\pi)Y_{1}^{\bm{d}(1)}\cdots Y_{m}^{\bm{d}(m)}M_{\alpha} \ | \ \pi\in\mathfrak{S}_{m}, \bm{d}\in\mathbb{Z}_{\geq 0}^{[m]}, \alpha\in\mathcal{P}\right\}, \]
where the associated structure constants belong to $\mathbb{Z}$.
\end{cor}

\begin{cor} \label{Cor:CenterFHm}
The center of $\mathsf{FH}_{m}$ is $\mathsf{Sym}[Y_{1},\dots,Y_{m}]\otimes\mathsf{FH}_{0}$, which is generated by the elements $E_{k}$ and the elementary symmetric polynomials $e_{k}(Y_{1},\dots,Y_{m})$ for all $k\geq 0$.
\end{cor}

\begin{proof}
By \emph{\Cref{Thm:FHmIsomorphicm}} it is clear that the result follows if $\mathsf{Sym}[y_{1},\dots,y_{m}]$ is the center of $\mathcal{H}_{m}$, which is shown in \cite[Theorem 3.3.1]{K05}.

\end{proof}

For any $m, n\in\mathbb{Z}_{\geq 0}$, let $\sigma_{n,m}\in\mathfrak{S}_{n+m}$ be the permutation given by the product of transpositions $(i,n+i)$ for each $i\in[m]$. In otherwords, in one-line notation we have that
\[ \sigma_{n,m}=(n+1)(n+2)\cdots(n+m)12\cdots m. \]
Let $(-)^{\sigma_{n,m}}:\mathbb{Z}\mathfrak{S}_{n+m}\rightarrow\mathbb{Z}\mathfrak{S}_{n+m}$ denote the $\mathbb{Z}$-linear extension of conjugation by $\sigma_{n,m}$. Let $\mathfrak{S}_{m}'$ denote the subgroup of $\mathfrak{S}_{n+m}$ of permutations on the set $\{n+1,\dots,n+m\}$. Then restricting  $(-)^{\sigma_{n,m}}$ to $Z_{n+m,m}$ yeilds an isomorphism between $Z_{n+m,m}$ and $(\mathbb{Z}\mathfrak{S}_{n+m})^{\mathfrak{S}_{m}'}:=\{z\in\mathbb{Z}\mathfrak{S}_{n+m} \ | \ za=az \text{ for all } a\in\mathfrak{S}_{m}'\}$. Recall that for any $i\in[m]$ we have that $\mathsf{pr}_{n+m,m}(Y_{i})=\sigma_{n,m}L_{i+m}\sigma_{n,m}^{-1}$, so twisting by $(-)^{\sigma_{n,m}}$ gives an epimorphism
\[ (-)^{\sigma_{n,m}}\circ\mathsf{pr}_{n+m,m}:\mathsf{FH}_{m}\rightarrow(\mathbb{Z}\mathfrak{S}_{n+m})^{\mathfrak{S}_{m}'}, \]
sending $Y_{i}\mapsto L_{n+i}$ and $E_{k}\mapsto e_{k}(L_{1},\dots,L_{n})$ where $L_{1},\dots,L_{n+m}$ are the Jucys-Murphy elements of $\mathbb{Z}\mathfrak{S}_{n+m}$. Therefore, feeding the above results through this epimorphism gives a uniform and alternative manner of proving the following:

\begin{cor}
The algebra $(\mathbb{Z}\mathfrak{S}_{n+m})^{\mathfrak{S}_{m}'}$ is generated by $s_{i}$, $L_{n+j}$, and $e_{k}(L_{1},\dots,L_{n})$ for $i\in[m-1]$, $j\in[m]$, and $n\geq k\geq 0$.
\end{cor}

\begin{cor}
The algebra $(\mathbb{Z}\mathfrak{S}_{n+m})^{\mathfrak{S}_{m}'}$ has an basis given by the set
\[ \left\{ \pi L_{n+1}^{\bm{d}(1)}\cdots L_{n+m}^{\bm{d}(m)}m_{\alpha}(L_{1},\dots,L_{n}) \ | \ \pi u^{\bm{d}}\overline{\alpha}\in\Lambda_{\leq n}(m) \right\}. \]
\end{cor}

\begin{cor}
The subalgebra $\mathsf{Sym}[L_{1},\dots,L_{n} | L_{n+1},\dots,L_{n+m}]$ of polynomials symmetric in the first $n$ and last $m$ Jucys-Murphy elements is central in $(\mathbb{Z}\mathfrak{S}_{n+m})^{\mathfrak{S}_{m}'}$.
\end{cor}

%%%%%%%%%%%%%%%%    REFERENCES   %%%%%%%%%%%%%%%%%%%%%%%%%%%%%%%%%%%%%%%%%%%%%%%%%%%%%%%%%%%


\begin{thebibliography}{99}

\bibitem[CD22]{CD22}
 S. Creedon and M. De Visscher,
 \textit{Defining an affine partition algebra},
 URL: \url{https://arxiv.org/abs/2110.08652}

\bibitem[Del04]{Del04}
  P. Deligne,
  \textit{La Categorie des Representations du Groupe Symetrique $S_{t}$, lorsque $t$ nest pas un Entier Naturel},
  2004,
  URL: \url{https://www.math.ias.edu/files/deligne/Symetrique.pdf} 

\bibitem[DEM13]{DEM13}
  S. Danz, H. Ellers, and J. Murray
  \textit{The centralizer of a subgroup in a group algebra},
   Proceedings of the Edinburgh Mathematical Society, 56(1), 49-56, 2013,
  URL: \url{https://doi.org/10.1017/S0013091512000077}

\bibitem[FH59]{FH59}
  H. Farahat and G. Higman,
  \textit{The centres of symmetric group rings},
   Proc. Roy. Soc. (A) 250 (1959), 212–221.
  URL: \url{https://doi.org/10.1098/rspa.1959.0060}

\bibitem[GJ94]{GJ94}
  I. Goulden and D. Jackson,
  \textit{Symmetric functions and Macdonald’s result for top connexion coefficients in the symmetric group},
  J. Algebra 166 (1994), no. 2, 364–378.
  URL: \url{https://doi.org/10.1006/jabr.1994.1157}

\bibitem[HS22]{HS22}
  N. Harman and A. Snowden,
  \textit{Oligomorphic groups and tensor categories},
  URL: \url{https://arxiv.org/abs/2204.04526}

\bibitem[Jucys74]{Jucys74}
  A. Jucys,
  \textit{Symmetric polynomials and the center of the symmetric group ring},
  Rep. Mathematical Phys. \textbf{5} (1974), no. 1, 107-112. MR 0419576 (54 7597),
  URL: \url{https://doi.org/10.1016/0034-4877(74)90019-6} 

\bibitem[Kho14]{Kho14}
  M. Khovanov,
  \textit{Heisenberg algebra and a graphical calculus},
  Fundamenta Mathematicae 225 (2014), 169-210,
  URL: \url{https://arxiv.org/abs/1009.3295}

\bibitem[K05]{K05}
 A. Kleshchev, 
 \textit{Linear and projective representations of the symmetric group}, 
 Cambridge Tract in Mathematics, series number 163, CUP, 2005.

\bibitem[Mac95]{Mac95}
 I. G. Macdonald, 
 \textit{Symmetric functions and Hall polynomials}, 
 2nd ed., Oxford University Press, 1995.

\bibitem[MO01]{MO01}
 A.I. Molev and G.I. Olshanski, 
 \textit{Degenerate affine Hecke algebras and centralizer construction for the symmetric groups}, 
 J. Algebra 237 (1) (2001) 302–341.

\bibitem[Murphy83]{Murphy83}
  G. E. Murphy,
  \textit{The idempotents of the symmetric group and Nakayama's conjecture},
  J. Algebra \textbf{81} (1983), no. 1, 258–265. MR 696137 (84k:20007),
  URL: \url{https://doi.org/10.1016/0021-8693(83)90219-3}

\bibitem[OV96]{OV96}
  A. Okounkov and A. Vershik,
  \textit{A new approach to the representation theory of symmetric groups}
  Selecta Math. (N.S.) \textbf{2} (1996), 581-605,
  URL: \url{https://doi.org/10.1007/s10958-005-0421-7}

\bibitem[Ryba21]{Ryba21}
  C. Ryba,
  \textit{Stable Centres I: Wreath Products}
  URL: \url{https://doi.org/10.48550/arXiv.2107.03752}

\bibitem[TW09]{TW09}
  J. Tysse and W. Wang,
  \textit{The centers of spin symmetric group algebras and Catalan numbers}
   J Algebr Comb (2009) 29: 175–193
  URL: \url{https://doi.org/10.1007/s10801-008-0128-1}

\bibitem[W03]{W03}
  W. Wang,
  \textit{The Farahat-Higman ring of wreath products and Hilbert schemes}
   Adv. Math., 187 (2) (2004), pp. 417-446
  URL: \url{https://doi.org/10.1016/j.aim.2003.09.003}

\bibitem[WW19]{WW19}
  J. Wan and W. Wang,
  \textit{Stability of the centers of group algebras of $GL_{n}(q)$}
   Adv. Math 349 (2019), pp. 749–780
  URL: \url{https://doi.org/10.1016/j.aim.2019.04.026}




\end{thebibliography}
\end{document}